%% file: v2.tex
\documentclass[11pt,a4paper]{amsart}
\input{top}

\title[On profinite rigidity amongst free-by-cyclic groups]{On profinite rigidity amongst free-by-cyclic groups I: the generic case}

\author{Sam Hughes}
\email{sam.hughes.maths@gmail.com; hughes@math.uni-bonn.de}
\address[S.~Hughes (current)]{Rheinische Friedrich-Wilhelms-Universit\"at Bonn, Mathematical Institute, Endenicher Allee 60, 53115 Bonn, Germany}
\author{Monika Kudlinska}
\email{mak74@cam.ac.uk}
\address[M. Kudlinska (current)]{Emmanuel College, St Andrew's Street, Cambridge CB2 3AP, UK}
\address[S. Hughes and M. Kudlinska (former)]{Mathematical Institute, Andrew Wiles Building, Observatory Quarter, University of Oxford, Oxford OX2 6GG, UK}

\date{\today}

\begin{document}
\begin{abstract}
We prove that amongst the class of free-by-cyclic groups, Gromov hyperbolicity is an invariant of the profinite completion. We show that whenever $G$ is a free-by-cyclic group with first Betti number equal to one, and $H$ is a free-by-cyclic group which is profinitely isomorphic to $G$, the ranks of the fibres and the characteristic polynomials associated to the monodromies of $G$ and $H$ are equal. We further show that for hyperbolic free-by-cyclic groups with first Betti number equal to one, the stretch factors of the associated monodromy and its inverse is an invariant of the profinite completion. We deduce that irreducible free-by-cyclic groups with first Betti number equal to one are almost profinitely rigid amongst irreducible free-by-cyclic groups.  We use this to prove that generic free-by-cyclic groups are almost profinitely rigid amongst free-by-cyclic groups. We also show similar results for \{universal Coxeter\}-by-cyclic groups.
\end{abstract}

\maketitle

\section{Introduction}
Two finitely generated groups $G$ and $H$ are said to be \emph{profinitely isomorphic} if they share the same isomorphism types of finite quotient groups.  It is a classical result that if two groups are profinitely isomorphic then they have the same profinite completion \cite{DixonFormanekPolandRibes1982}.  For a class $\calc$ of finitely generated residually finite groups, a group $G\in\calc$ is \emph{profinitely rigid in $\calc$} if any group $H$ in $\calc$ profinitely isomorphic to $G$ is in fact isomorphic to $G$.  Similarly, we say $G$ is \emph{almost profinitely rigid in $\calc$} if there are at most finitely isomorphism types of groups $H$ in $\mathcal{C}$ profinitely isomorphic to $G$.

There exists a large body of work investigating profinite rigidity of $3$-manifold groups.  For example, deep work of Bridson--McReynolds--Reid--Spitler shows that there are hyperbolic $3$-manifolds which are profinitely rigid amongst all finitely generated residually finite groups \cite{BridsonMcReynoldsReidSpitler2020}, with more examples constructed in \cite{CheethamWest2022} and \cite{BridsonReid2022}.  On the other hand, there exist Anosov torus bundles and periodic closed surface bundles with non-isomorphic but profinitely isomorphic fundamental groups \cite{Stebe1972,Funar2013,Hempel2014}.

Significant progress has been made on the problem of profinite rigidity \emph{within} the class of 3-manifolds. A key step in showing that various classes and properties of 3-manifolds are invariants of the profinite completion is to establish the profinite invariance of fibring. In this vein, and in order to deduce results about the profinite completion of knot groups, Bridson--Reid studied profinite invariants of compact 3-manifolds with boundary and first Betti number equal to one, in particular showing that fibring and the rank of the fibre is a profinite invariant of such 3-manifolds \cite{BridsonReid2020}. At the same time, Boileau--Friedl tackled the problem of profinite invariants of knot groups by showing that fibring is an invariant of 3-manifolds whose profinite completions are related by a particular type of isomorphism, called a \emph{regular isomorphism} \cite{BoileauFriedl2020}. Finally, Jaikin-Zapirain showed that being fibred is a profinite invariant of all $3$-manifold groups \cite{Jaikin2020}, and this was generalised to all LERF groups in \cite{HughesKielak2022}.

Another crucial element is the work Wilton--Zalesskii on profinite detection of Thurston geometries \cite{WiltonZalesskii2017} and of Wilkes and Wilton--Zalesskii on profinite invariance of various decompositions of 3-manifolds \cite{Wilkes2018,Wilkes2018b,WiltonZalesskii2019}. The case of Seifert fibred manifolds was entirely solved by Wilkes, who proved that these are almost profinitely rigid in the class of all 3-manifold groups \cite{Wilkes2017}.  Graph manifolds have received much attention too \cite{WiltonZalesskii2010,Wilkes2018b,Wilkes2019}. Most recently, Liu proved the spectacular theorem that finite volume hyperbolic $3$-manifold groups are almost profinitely rigid \cite{Liu2023}.  Other results have also been obtained, e.g. \cite{BridsonReidWilton2017,WiltonZalesskii2017b,BoileauFriedl2020,Zalesskii2022,Liu2023b}.

We say a group $G$ is \emph{free-by-cyclic} if it contains a normal subgroup $N \trianglelefteq G$ which is isomorphic to a non-trivial free group of finite rank $F_n$, and such that $G/ N \cong \Z$. We will almost always think of a free-by-cyclic group as a pair $(G, \varphi)$, where $\varphi \in \mathrm{Hom}(G; \Z)$ is an epimorphism which gives rise to a short exact sequence
\[1 \to F_n \to G \xrightarrow{\varphi}  \Z \to 1.\]
Since any such short exact sequence splits, one can realise a free-by-cyclic group as the semi-direct product $G \cong F_n \rtimes_{\Phi} \Z$, for some $\Phi \in \mathrm{Out}(F_n)$ which we refer to as the \emph{monodromy} of the splitting. Conversely, given a semi-direct splitting $G \cong F_n \rtimes_{\Phi} \Z$ there's an associated character $\varphi \colon G \to \Z$ which maps the normal free factor to zero, and the stable letter (with respect to any choice of representative of $\Phi$) to the generator 1 of $\Z$. We call this the \emph{induced character} of the splitting $F_n \rtimes_{\Phi} \Z$.

Free-by-cyclic groups form a well-studied class which has been shown to exhibit many key properties; these include residual finiteness \cite{Baumslag1971}, quadratic isoperimetric inequality \cite{BridsonGroves2010}, and the property of being large \cite{Button2013}. Further, it is known that hyperbolic free-by-cyclic groups are cubulable \cite{HagenWise2015} and thus virtually compact special in the sense of Haglund--Wise \cite{HaglundWise2008}, and more generally that all free-by-cyclic groups which do not virtually split as a direct product admit non-elementary acylindrical actions on hyperbolic spaces \cite{GenevoisHorbez2021}. Despite this, there are still many open questions in this area, most notably on the subject of rigidity, even when one considers only rigidity amongst the class of free-by-cyclic groups.

Our goal in writing this paper is to investigate profinite rigidity amongst free-by-cyclic groups. The study of profinite invariants of free-by-cyclic groups saw its inception in the work of Bridson--Reid  \cite{BridsonReid2020}. Although the aim of their work was to prove results about fibred knot complements, their methods apply more generally and are later used by Bridson--Reid--Wilton \cite{BridsonReidWilton2017} to show profinite rigidity amongst the groups of the form $F_2 \rtimes \Z$.  

Whilst we draw inspiration from the results in the $3$-manifold setting, the problem for free-by-cyclic groups is significantly more subtle. This stems in part from the lack of a sufficient $\mathrm{Out}(F_n)$-analogue of the Nielsen--Thurston decomposition for homeomorphisms of finite-type surfaces. One artefact of this is that we frequently have to restrict our attention to the class of \emph{irreducible} free-by-cyclic groups, that is free-by-cyclic groups which admit irreducible monodromy. Recall that an outer automorphism $\Phi \in \mathrm{Out}(F_n)$ is \emph{irreducible} if there does not exist a free splitting \mbox{$F_n=A_1 \ast \ldots \ast A_k \ast B$,} where $A_1 \ast \ldots \ast A_k$ is non-trivial, and such that $\Phi$ permutes the conjugacy classes of the factors $A_i$. By the work of Mutanguha \cite{Mutanguha2021}, for any two realisations of $G$ as a free-by-cyclic group, $G \cong F_n \rtimes_{\Phi} \Z \cong F_m \rtimes_{\Psi} \Z$, the monodromy $\Phi$ is irreducible if and only if $\Psi$ is. 

Our first result is analogous to Liu's theorem with the additional hypotheses that $b_1(G)=1$ and restricting to the class of irreducible free-by-cyclic groups.  The first hypothesis is due to the fact that we do not have a method to establish $\widehat{\Z}$-regularity (see \Cref{sec:regularity} for a definition) without an analogous result to the main theorems in \cite{FriedlVidussi2008,FriedlVidussi2011annals} --- this is one of the main technical steps in Jaikin-Zapirain's and Liu's results.  The second hypothesis arises since, although we can show that hyperbolicity of free-by-cyclic groups is a profinite invariant, we are currently unable to show the same holds true for irreducibility.

\medskip
\begin{duplicate}[\Cref{thmx:Irr}]
    Let $G$ be an irreducible free-by-cyclic group.  If $b_1(G)=1$, then $G$ is almost profinitely rigid amongst irreducible free-by-cyclic groups.
\end{duplicate}
\medskip

\subsection{Profinite invariants}
The next theorem is somewhat more technical.  We will not include definitions of the invariants in the introduction, but many of them will be familiar to experts and they are scattered throughout the paper.  Note that the result actually holds in the more general setting of a $\widehat{\Z}$-regular isomorphism (the specific results stated throughout the paper comprising \Cref{thmx:invariants} are stated in this generality, in fact we provide a restatement of \Cref{thmx:invariants} later in the text in this generality). 

We point out the general fact that the first Betti number of any finitely generated discrete group is an invariant of its profinite completion.

\medskip
\begin{duplicate}[\Cref{thmx:invariants}]
    Let $G=F\rtimes_{\Phi} \Z$ be a free-by-cyclic group with induced character $\varphi \colon G \to \Z$. If $b_1(G)=1$, then the following properties are determined by the profinite completion $\widehat{G}$ of $G$: \begin{enumerate}
        \item the rank of $F$;
        \item the homological stretch factors $\{\nu^+_G,\nu^-_G\}$;
        \item the characteristic polynomials $\{\Char{\Phi^+},\Char{\Phi^-}\}$ of the action of $\Phi$ on $H_1(F;\QQ)$;
        \item for each representation $\rho\colon G\to \GL(n, \QQ)$ factoring through a finite quotient, the twisted Alexander polynomials $\{\Delta^{\varphi,\rho}_n,\Delta^{-\varphi,\rho}_n\}$ and the twisted Reidemeister torsions $\{\tau^{\varphi,\rho},\tau^{-\varphi,\rho}\}$ over $\QQ$.
    \end{enumerate}
    Moreover, if $G$ is conjugacy separable, (e.g. if $G$ is hyperbolic), then $\widehat G$ also determines the Nielsen numbers and the homotopical stretch factors $\{\lambda^+_G,\lambda^-_G\}$.
\end{duplicate}
\medskip

We note that our Theorem~B(1) was already known by the work of Bridson--Reid \cite[Lemma 3.1]{BridsonReid2020}. 

The reason for obtaining a set of invariants corresponding to $\Phi$ and $\Phi^{-1}$ is that the dynamics of $\Phi$ and $\Phi^{-1}$ can be different.  Indeed, this is somewhat a feature of free-by-cyclic groups rather than a bug.  A large technical hurdle in this work was overcoming this phenomenon which cannot occur for $3$-manifolds.

We also obtain a complete geometric picture \`a la Wilton--Zalesskii in the case of hyperbolic free-by-cyclic groups.

\medskip
\begin{duplicate}[\Cref{thmx:hyperbolicity}]
     Let $G_A$ and $G_B$ be profinitely isomorphic free-by-cyclic groups. Then $G_A$ is Gromov hyperbolic if and only if $G_B$ is Gromov hyperbolic.
\end{duplicate}
\medskip

\subsection{Almost profinite rigidity and applications}
We will now explain how to apply \Cref{thmx:Irr}, \Cref{thmx:invariants}, and \Cref{thmx:hyperbolicity} to various classes of free-by-cyclic groups to obtain strong profinite rigidity phenomena.

\subsubsection{Super irreducible free-by-cyclic groups}
We say that a free-by-cyclic group $G$ is \emph{super irreducible}, if $G \cong F_n \rtimes_{\Phi} \Z$ and the integral matrix $M\colon H_1(F_n;\QQ)\to H_1(F_n;\QQ)$ representing the action of $\Phi$ on $H_1(F_n;\QQ)$ satisfies the property that no positive power of $M$ maps a proper subspace of $H_1(F_n;\QQ)$ into itself.  Note that this immediately implies $b_1(G)=1$ because 
\[ H_1(G; \QQ) \cong \left( H_1(F_n; \QQ) / \mathrm{Im}(M - \mathrm{Id}) \right)\oplus \QQ ,\]
and since $M$ is super irreducible, $\ker (M- \mathrm{Id}) = \{0\}$. Super irreducibility also implies $G$ is irreducible by \cite[Theorem~2.5]{GerstenStallings1991}. 

An example of a super irreducible free-by-cyclic group is  whenever the characteristic polynomial of $M$ is a \emph{Pisot--Vijayaraghavan polynomial}, namely, it is monic, it has exactly one root (counted with multiplicity) with absolute value strictly greater than one, and all other roots have absolute value strictly less than one \cite{GerstenStallings1991}.

\medskip
\begin{duplicate}[\Cref{corx:PV}]
Let $G$ be a super irreducible free-by-cyclic group.  Then every free-by-cyclic group profinitely isomorphic to $G$ is super irreducible.  In particular, $G$ is almost profinitely rigid amongst free-by-cyclic groups.
\end{duplicate}
\medskip

\subsubsection{Random free-by-cyclic groups}

Fix $n \geq 2$ and let $S$ be a finite generating set of $\mathrm{Out}(F_n)$. For any $l \geq 1$, define $\mathcal{H}_{l,n}$ to be the set of all free-by-cyclic groups 
$G$ which admit a splitting $G \cong F_n \rtimes_{\Phi} \Z$, where $\Phi$ can be expressed as a word of length at most $l$ in $S$. We say that for a \emph{random} free-by-cyclic group the property $P$ holds \emph{asymptotically almost surely}, or that a \emph{generic free-by-cyclic group satisfies property $P$}, if 
\[ \frac{\#\{G \in \mathcal{H}_{l,n} \mid G \text{ satisfies property }P\} }{\#\mathcal{H}_{l,n}} \to 1 \text{ as }l\to \infty.\]

We now state the result alluded to in the title of the paper.

\medskip
\begin{duplicate}[\Cref{corx:generic}]
    Let $G$ be a random free-by-cyclic group.  Then, asymptotically almost surely $G$ is almost profinitely rigid amongst free-by-cyclic groups.
\end{duplicate}
\medskip

\subsubsection{Low rank fibres}
When the fibre of the free-by-cyclic group has rank two or three we are able to obtain rigidity statements within the class of all free-by-cyclic groups.

\medskip
\begin{duplicate}[\Cref{corx:F3}]
Let $G=F_3\rtimes\Z$.  If $G$ is hyperbolic and $b_1(G)=1$, then $G$ is almost profinitely rigid amongst free-by-cyclic groups.
\end{duplicate}
\medskip

Note in the next statement we see that $G$ is uniquely determined.

\medskip
\begin{duplicate}[\Cref{corx:F2}]
     Let $G=F_2\rtimes\Z$.  If $b_1(G)=1$, then $G$ is profinitely rigid amongst free-by-cyclic groups.
\end{duplicate}
\medskip

\subsubsection{Profinite conjugacy}
Our next result investigates conjugacy in $\Out(\widehat{F}_n)$ and is somewhat analogous to \cite[Theorem~1.2]{Liu2023b}.  We say two outer automorphisms $\Psi$ and $\Phi$ of $F_n$ are \emph{profinitely conjugate} if they induce a conjugate pair of outer automorphisms in $\Out(\widehat{F}_n)$.  In this setting we have no assumption on the action of $\Psi$ or $\Phi$ on the homology of $F_n$.

\medskip
\begin{duplicate}[\Cref{thmx:procongruence}]
Let $\Psi\in\Out(F_n)$ be atoroidal.  If $\Phi\in\Out(F_n)$ is profinitely conjugate to $\Psi$, then $\Phi$ is atoroidal and $\{\lambda_\Psi,\lambda_{\Psi^{-1}}\}=\{\lambda_\Phi,\lambda_{\Phi^{-1}}\}$.  In particular, if $\Psi$ is additionally irreducible, then there are only finitely many $\Out(F_n)$-conjugacy classes of irreducible automorphisms which are conjugate with $\Psi$ in $\Out(\widehat{F}_n)$.
\end{duplicate}
\medskip

\subsubsection{Automorphisms of universal Coxeter groups} 
Finally, we extend our results to the setting of universal Coxeter groups. A group $G$ is \emph{\{universal Coxeter\}-by-cyclic} if it splits as a semi-direct product $W_n \rtimes \Z$ where $W_n = \bigast_{i=1}^n \Z /2$ is the free product of $n$ copies of $\Z / 2$. A \emph{free basis} of $W_n$ is a generating set for $W_n$ such that each element has order 2.

Let $K \leq W_n$ be the unique torsion-free subgroup of index 2. For any choice of free basis for $W_n$, $K$ is the kernel of the homomorphism $W_n \to \Z/2$ which maps every free generator of $W_n$ to 1. We note that $K$ is characteristic and it is isomorphic to the free group of rank $n-1$. 

Fix a free basis of the free group $F_n$ of rank $n$, and let $\iota \in \mathrm{Aut}(F_n)$, denote the automorphism which inverts each basis element. Let $[\iota]$ be the image of $\iota$ in $\mathrm{Out}(F_n)$. Following \cite{BregmanFullarton2018}, we define the group of \emph{hyperelliptic outer automorphisms} of $F_n$, denoted by $\mathrm{HOut}(F_n)$, to be the centraliser of $[\iota]$ in $\mathrm{Out}(F_n)$.

\medskip
\begin{duplicate}[\Cref{thmx:CoxeterStretchInvariance}]
    Let $G = W \rtimes \Z$ be a \{universal Coxeter\}-by-cyclic group. Then the rank of the fibre $W$ is an invariant of $\widehat{G}$. 

    Suppose that all free-by-cyclic groups with monodromy in $\mathrm{HOut}(F_n)$ for some $n$, are conjugacy separable. Then $\widehat{G}$ determines the the stretch factors $\{
    \lambda^{+}, \lambda^{-}\}$ associated to the monodromy of the splitting $W \rtimes \Z$.
\end{duplicate}
\medskip

\subsection{Some unanswered questions}
While we began in earnest to transport the programme of profinite rigidity amongst $3$-manifold groups to free-by-cyclic groups, we have perhaps raised as many questions as answers.  We will highlight some key questions that we have encountered and hope to answer in the future.
Perhaps the most pressing issue is that of $\widehat{\Z}$-regularity. 

\begin{question}\label{q.Zregular}
    Is every profinite isomorphism of free-by-cyclic groups $\widehat{\Z}$-regular?
\end{question}

 One may hope to answer the previous question as in \cite{Liu2023}, but using the agrarian polytope \cite{HennekeKielak2021,Kielak2020polytopes} in place of the Thurston polytope.  The key issue is that we do not have the $\TAP_1$ property for free-by-cyclic groups (for $3$-manifolds this is a deep result of Friedl--Vidussi \cite{FriedlVidussi2008,FriedlVidussi2011annals}).  The reader is referred to  \cite[Definition~3.1]{HughesKielak2022} for the definition due to its technical nature.
 
\begin{question}\label{q.TAP1}
    Is every free-by-cyclic group $G$ in $\mathsf{TAP}_1(\FF)$ for $\FF\in\{\QQ,\FF_p\}$ with $p$ prime? 
\end{question}

The other somewhat obvious question is whether irreducibility is a profinite invariant.  We expect this to be the case (at least amongst hyperbolic free-by-cyclic groups).

\begin{question}\label{q.irreducibility}
    Is being irreducible a profinite invariant amongst free-by-cyclic groups?
\end{question}

Our final question is motivated by \cref{thmx:CoxeterStretchInvariance}.

\begin{question}\label{q.conjugacySep}
    Is it true that for every hyperelliptic outer automorphism $\Phi \in \mathrm{HOut}(F_n)$, the mapping torus $G = F_n \rtimes_{\Phi} \Z$ is conjugacy separable?
\end{question}

\subsection{Structure of the paper}
In \Cref{prelim:autos} we recall the necessary background on free group automorphisms and free-by-cyclic groups and prove a number of results we will need throughout the paper.  

In \Cref{TopRep} we recall the definition of a topological representative of a free group automorphism, its stretch factor and the various definitions of irreducibility we will need.  We include a proof that there are at most finitely many equivalence classes of irreducible topological representatives such that the graph has rank $n$ and the stretch factor is at most some positive real number $C>1$ (\Cref{min}).  

In \Cref{sec:generic} we study generic outer automorphisms of free groups and prove that a generic free-by-cyclic group has first Betti number equal to one and is super irreducible (\Cref{generic_properties}).  

In \Cref{sec:Nielson} we relate the Nielsen numbers of an outer automorphism of a free group to the stretch factor of the outer automorphism.

In \Cref{sec:hyp} we study certain subgroup separability properties of free-by-cyclic groups.  In particular, we show that every abelian and every free-by-cyclic subgroup is separable (\Cref{fully_separable}).  We combine this with results of Wilton--Zalesskii \cite{WiltonZalesskii2017} to prove \Cref{thmx:hyperbolicity} from the introduction.

In \Cref{sec:AP} we recall the definitions of twisted Alexander polynomials and twisted Reidemeister torsions.  We establish a number of facts about twisted Alexander polynomials which we will use later in the paper.  Our main new contribution is \emph{a complete calculation of the zeroth twisted Alexander polynomials} over $\QQ$ for any finitely generated group (\Cref{zeroth AP palindromic}), as well as a formula for the twisted Reidemeister torsion of a free-by-cyclic group in terms of the twisted Alexander polynomials.

In \Cref{sec:regularity} we recall the notion of a matrix coefficient module and a $\widehat{\Z}$-regular isomorphism.  The main reason for this section is to allow us to work in the generality of a $\widehat{\Z}$-regular isomorphism. This means that if one established a positive answer to \Cref{q.Zregular} then one could apply the results in this paper without any further modifications.  

At this stage, we establish some notation.  Let $G_A$ be a free-by-cyclic group with character $\psi$ and fibre subgroup $F_A$.  Also let $G_B$ be a free-by-cyclic group with character $\varphi$ and fibre subgroup $F_B$.  Let $\Theta\colon \widehat{G}_A\to\widehat{G}_B$ be a $\widehat{\Z}$-regular isomorphism (see \Cref{defn:Zhatreg}).  Our final result of the section is that $F_A\cong F_B$ (\Cref{fibre iso}).

In \Cref{sec:Rtorsion} we set out to prove profinite invariance of Reidemeister torsion over $\QQ$ twisted by representations of finite quotients for $G_A$ and $G_B$.  Our strategy is parallel to that of Liu \cite[Section~7]{Liu2023}, however due to the extra complexity of free-by-cyclic groups we have to invoke extra results about twisted Alexander polynomials of free-by-cyclic groups established in \Cref{sec:AP}.  In \Cref{sec:AP.profinite} we prove profinite invariance of the twisted Alexander polynomials although we work in the more general setting of $\{$good type $\mathsf{F}\}$-by-$\Z$ groups and $\widehat{\Z}$-regular isomorphisms.  In \Cref{sec:profRtorsion} we establish the profinite invariance of twisted Reidemeister torsion for $G_A$ and $G_B$.  In \Cref{sec.homostretch} we prove that the homological stretch factors $\{\nu_A,\nu_{A}^{-}\}$ and $\{\nu_B,\nu_{B}^{-}\}$ are equal.

In \Cref{sec:profNielson}, under the assumption of conjugacy separability of $G_A$ and $G_B$ we prove that the homotopical stretch factors $\{\lambda_A,\lambda_A^{-}\}$ and $\{\lambda_B,\lambda_{B}^{-}\}$ are equal.  Again our strategy is largely motivated by \cite[Section~8]{Liu2023}.  The key difference is that for a fibred character $\chi$ on a finite volume hyperbolic $3$-manifold the stretch factors of $\chi$ and $\chi^{-1}$ are the same.  This is not true for free-by-cyclic groups where we must deal with \emph{both directions at once}\footnote{John Coltrane - The Lost Album.} and so our main work is resolving this issue.

Combining the major results up to this point proves \Cref{thmx:invariants}.

In \Cref{sec.proofmain} we prove \Cref{thmx:Irr}.  In the hyperbolic case this is a corollary of \Cref{thmx:invariants} and the fact that hyperbolic free-by-cyclic groups are virtually special and hence conjugacy separable.  In the general case we apply a result of Mutanguha \cite{Mutanguha2021} and train track theory to deduce the conjugacy separability we require.  We then go on to deduce Corollaries~\ref{corx:PV}-\ref{corx:F2}.

In \Cref{sec.proconjugacy} we prove \Cref{thmx:procongruence}.  This is really an easy consequence of \Cref{thmx:invariants} once we transport a result of Liu \cite[Proposition~3.7]{Liu2023b} on profinite conjugacy of mapping class groups to the $\Out(F_n)$ setting.

Finally, in \Cref{sec:Wn} we prove results on profinite invariants and profinite almost rigidity of \{universal Coxeter\}-by-cyclic groups. To do so, we start by establishing notation and recalling background on morphisms of graphs of groups in \Cref{sec:graphofgroups}. The purpose of \Cref{sec:traintrackWn} is to relate the theory of train track representatives of elements in $\mathrm{Out}(W_n)$ with Nielsen fixed point theory. We also prove a lemma on irreducibility of covers of directed graphs and use this to relate the stretch factor of an outer automorphism $\Phi \in \mathrm{Out}(W_n)$ with the stretch factor of the free group automorphism obtained by restricting $\Phi$ to a free characteristic subgroup of $W_n$. In the final \Cref{sec:ProfiniteCoxeter} we combine results from previous sections to prove \Cref{thmx:CoxeterStretchInvariance}.

\subsection*{Acknowledgements}
We would like to extend a big thank you to Naomi Andrew, Martin Bridson, and especially Dawid Kielak for a number of helpful conversations.  

We would like to thank Martin Bridson, Francesco Fournier Facio, Dawid Kielak, Jean Pierre Mutanguha, and Sam Taylor for helpful comments on an earlier version of this article.

We would like to thank the anonymous referees for carefully reading the article and providing a number of helpful comments.

This work has received funding from the European Research Council (ERC) under the European Union's Horizon 2020 research and innovation programme (Grant agreement No. 850930). The second author was supported by an Engineering and Physical Sciences Research Council studentship (Project Reference 2422910).

\newpage
\subsection{Notation}
We include a table of notation for the aid of the reader.

\begin{table}[h]
    \centering
    \begingroup
    \renewcommand{\arraystretch}{1.2}
\begin{tabular}{|c|p{10cm}|}
\hline
Symbol & Definition\\
\hline
    $F_n$ & Free group of rank $n$\\
    $W_n$ & Universal Coxeter group of rank $n$, that is, $\aster_{i=1}^n\Z/2$\\
    \hline
    $\Gamma$ & Graph \\
    $(\Gamma, \mathcal{G})$, $\mathcal{G}$ & Graph of groups \\
    $X_{\mathcal{G}}$ & Graph of spaces \\
    $(f, f_X), f$ & Morphism of graphs of groups \\
     \hline
    $\psi$, $\varphi$, $\psi_A$, $\varphi_B$ & Character of a free-by-cyclic group\\
    $(G_A,\psi$), $(G_B,\varphi)$ & Free-by-cyclic group \\
    $F$, $F_A$, $F_B$ & Fibre subgroup \\
    $\Psi$, $\Phi$ & Outer automorphisms (of $F_n$ or $W_n$)\\
    $f$, $f_A$, $f_B$ & Train track \\
    $\mathrm{Orb}_m(f)$  & Set of $m$-periodic orbits of $f$\\
    $N_m(f)$ & $m$th Nielsen number of $f$\\
    $\lambda$, $\lambda_f$, $\lambda_\Psi$ & Homotopical stretch factor (of $f$ or $\Psi$) \\
    $\nu$, $\nu_f$, $\nu_\Psi$ & Homological stretch factor (of $f$ or $\Psi$)\\
    \hline
    $R$ & Unique factorisation domain\\
    $R^\times$ & Units of $R$\\
    $\Delta_{R,n}^{\varphi,\alpha}$ & $n$th Alexander polynomial of $\varphi$ twisted by $\alpha$ over $R$\\
    $\tau_R^{\varphi,\alpha}$ & Reidemeister torsion of $\varphi$ twisted by $\alpha$ over $R$\\
    \hline
    \multicolumn{2}{|p\textwidth|}{In some contexts we will drop the $R$ from the previous notations and replace it with a group $G$ for clarity, that is, $\Delta_{G,n}^{\varphi,\alpha}$ and $\tau_G^{\varphi,\alpha}$ or even $\tau_{G,R}^{\varphi,\alpha}$}\\
    \hline
    $\alpha$, $\beta$, $\gamma$ & Finite quotients \\
    $Q$ & Image of a finite quotient \\
    $\rho$, $\sigma$ & Representation of a group\\
    $\chi_\rho$ & Character of the representation $\rho$\\
    $\gamma^\ast(\sigma)$ & Pullback representation of $\sigma$ along $\gamma$\\
    $\mathbf{1}$ & The trivial representation\\
    \hline
    $\Theta$ & Profinite isomorphism\\
    $\MC(\Theta)$ & Mapping coefficient module\\
     $\Theta_\ast^\epsilon$, $\Theta^\ast_\epsilon$ & $\epsilon$-specialisation of $\Theta$\\
     $\mu$ & Unit of $\widehat{\Z}$\\
     \hline
\end{tabular}
\endgroup
    \caption{Table of notation.}
    \label{tab:notation}
\end{table}

\pagebreak

\section{Preliminaries on free group automorphisms}\label{prelim:autos}

\subsection{Topological representatives}\label{TopRep}

The contents of this subsection largely derive from the work of Bestvina--Handel in \cite{BestvinaHandel1992}. Let $n \geq 2$ and $\Phi \in \mathrm{Out}(F_n)$ be an outer automorphism of $F_n$.  A \emph{topological representative} of $\Phi$ is a tuple $(f, \Gamma)$, where $\Gamma$ is a connected graph with $\pi_1(\Gamma)\cong F_n$, and $f \colon \Gamma \to \Gamma$ is a homotopy equivalence which induces the outer automorphism $\Phi$. Furthermore, $f$ preserves the set of vertices of $\Gamma$ and it is locally injective on the interiors of the edges of $\Gamma$. A topological representative $f$ is said to be a \emph{train track} if every positive power of $f$ is locally injective on the interiors of edges.

Fix an ordering of the edges of $\Gamma$. The \emph{incidence matrix} $A$ of $f$ is the matrix with entries $a_{ij}$, where $a_{ij}$ is the number of occurrences of the unoriented edge $e_j$ in the edge-path $f(e_i)$.

Recall that a non-negative integral $n$-by-$n$ square matrix $M$ is said to be \emph{irreducible}, if for any $i,j \leq n$, there exists some $k\geq 1$ such that the $(i,j)$-th entry of $M^k$ is positive.

Let $(f, \Gamma)$ be a topological representative. A \emph{filtration of length l} of $(f, \Gamma)$ is a sequence of subgraphs
\begin{equation}\label{filtration} \emptyset = \Gamma_0 \subseteq \Gamma_1 \subseteq \ldots \subseteq \Gamma_l = \Gamma,\end{equation}
so that $f(\Gamma_i) \subseteq \Gamma_i$ for all $i$. The closure $S_i = \mathrm{Cl}(\Gamma_i \setminus \Gamma_{i-1})$ is called the \emph{i}th \emph{stratum} of the filtration. Re-order the edges of $\Gamma$ so that whenever $i < j$, the edges in $\Gamma_i$ precede the edges in $\Gamma_j$.  The filtration is said to be \emph{maximal} if the square submatrix $A_i$ of the incidence matrix $A$ which corresponds to the $i$-th stratum is either the zero matrix, or it is irreducible. It is a standard fact that any topological representative admits a maximal filtration which is unique up to reordering of the strata. If $(f, \Gamma)$ admits a maximal filtration of length one then we say that $(f, \Gamma)$ is \emph{irreducible.}

By the Perron--Frobenius theorem (see Chapter~2 in \cite{Seneta2006}), if $A_i$ is the submatrix of the incidence matrix $A$ of $(f, \Gamma)$ which corresponds to an irreducible stratum $S_i$, then the spectral radius $\rho(A_i)$ of $A_i$ is an eigenvalue of $A_i$, which is known as the \emph{Perron--Frobenius eigenvalue} of $A_i$. Furthermore, $\rho(A_i) \geq 1$ and equality holds exactly when $A_i$ is a permutation matrix. We call $S_i$ an \emph{exponentially-growing stratum} if its Perron--Frobenius eigenvalue is strictly greater than one. An edge $e$ of $\Gamma$ is said to be \emph{exponentially growing} if it is contained in some exponentially growing stratum. For a topological representative $(f, \Gamma)$, we write $\lambda_f$, (or $\lambda$ if there is no potential for confusion), to denote the maximal Perron--Frobenius eigenvalue taken over all the non-zero strata of the maximal filtration of $(f, \Gamma)$, and we call it the \emph{(homotopical) stretch factor} of $f$.

A subgraph is \emph{non-trivial} if it has a component which is not a vertex. An outer automorphism $\Phi \in \mathrm{Out}(F_n)$ is \emph{irreducible}, if every topological representative $(f, \Gamma)$ of $\Phi$, where $\Gamma$ has no valence-one vertices and no non-trivial $f$-invariant forests, is irreducible. A free-by-cyclic group $G$ is \emph{irreducible}, if $G$ admits a splitting $G \cong F_n \rtimes_{\Phi} \mathbb{Z}$, with $\Phi \in \mathrm{Out}(F_n)$ irreducible. Note that by \cite{Mutanguha2021}, if $G$ is irreducible then the monodromy associated to every fibred splitting of $G$ is an irreducible outer automorphism.

The \emph{stretch factor} of an irreducible outer automorphism $\Phi$ is the infimum of the stretch factors of the irreducible topological representatives of $\Phi$. By the proof of Theorem~1.7 in \cite{BestvinaHandel1992}, the infimum is realised. We will write $\lambda(\Phi)$ to denote the stretch factor of $\Phi$.

\begin{lemma}\label{min}
Let $n \geq 2$ and $C > 1$. There exist at most finitely many conjugacy classes of irreducible elements in $\mathrm{Out}(F_n)$ with stretch factor at most $C$.
\end{lemma}

\begin{proof}

    Let $\mathrm{CV}_n$ denote the Culler--Vogtmann Outer space. For $\epsilon > 0$, write $\mathrm{CV}_n(\epsilon)$ to denote the \emph{thick part} of $\mathrm{CV}_n$, which is defined as the set of all metric graphs $\Gamma$ in $\mathrm{CV}_n$ such that the length of every loop $\alpha$ in $\Gamma$ satisfies $\ell_{\Gamma}(\alpha) \geq \epsilon$. We consider $\mathrm{CV}_n$ as a metric space with the Lipschitz metric. 

     Let $\{\Phi_i\}_{i\in I}$ be a collection of irreducible elements in $\mathrm{Out}(F_n)$ which are non-pairwise conjugate, and such that $\lambda(\Phi_i) \leq C$ for each $i \in I$. Suppose first that $\lambda(\Phi_i) = 1$ for all $i \in I$. Then each $\Phi_i$ has finite order in $\mathrm{Out}(F_n)$. Every finite order element in $\mathrm{Out}(F_n)$ is induced by a periodic automorphism of a graph with no valence-one and valence-two vertices. In particular, there are finitely many finite order elements in $\mathrm{Out}(F_n)$ and hence $I$ is finite.

     Suppose now that some $\Phi_i$ has infinite order. Without loss of generality, we may assume that every $\Phi_i$ has infinite order. Let $\epsilon = 1/ ((3n-3)(C+1)^{3n-2})$. By \cite[Proposition~2.14]{FrancavigliaMartinoSyrigos2021}, each axis of $\Phi_i$ is contained in the $\epsilon$-thick part $\mathrm{CV}_n(\epsilon)$. 
     
     Since action of $\mathrm{Out}(F_n)$ on the thick part $\mathrm{CV}_n(\epsilon)$ is cocompact, there exists some compact subset $K \subseteq \mathrm{CV}_n(\epsilon)$ such that $\bigcup_{g \in \mathrm{Out}(F_n)} g\cdot K = \mathrm{CV}_n(\epsilon)$. Thus, for each $i \in I$ there is an element $\Psi_i \in \mathrm{Out}(F_n)$ which is conjugate to $\Phi_i$ and such that $\mathrm{Axis}(\Psi_i) \cap K \neq \emptyset$. Let $N_{\log C}(K)$ denote the $(\log C)$-neighbourhood of $K$ in $\mathrm{CV}_n(\epsilon)$. Then, $\Psi_i \cdot N_{\log C}(K) \cap N_{\log C}(K) \neq \emptyset$ for all $i \in I$. Since the thick part $\mathrm{CV}_n(\epsilon)$ is proper, we have that $N_{\log C}(K)$ is a compact subset. Hence, since the action of $\mathrm{Out}(F_n)$ on $\mathrm{CV}_n(\epsilon)$ is proper, it must be the case that $I$ is finite.
\end{proof}

A \emph{bounded} topological representative $(f, \Gamma)$ of $\Phi \in \mathrm{Out}(F_n)$ is such that the number of exponentially growing strata is bounded by $3n - 3$, and each exponential stratum stretch factor which is the Perron--Frobenius eigenvalue of an irreducible square matrix of dimensions bounded above by $3n -3$. For a general outer automorphism $\Phi$, we define the stretch factor $\lambda(\Phi)$ of $\Phi$ to be the infimum, taken over all the bounded topological representatives $(f, \Gamma)$ of $\Phi$, of $\lambda_{\mathrm{max}}$, where $\lambda_{\mathrm{max}}$ denotes the maximum stretch factor of the non-zero strata in a maximal filtration of $(f, \Gamma)$. The infimum $\Lambda(\Phi)$ is realised by a bounded relative train track representative $(f, \Gamma)$ (see \cite[p.37]{BestvinaHandel1992}).

\subsection{Generic elements of \texorpdfstring{$\mathrm{Aut}(F_n)$}{Aut(Fn)}}\label{sec:generic}

Fix a finite generating set $S$ of $\mathrm{Aut}(F_n)$. For each $l \geq 1$, let $\mathcal{W}_{l,n}$ denote the set of reduced words of length $l$ in $S$. We say that a \emph{random element of $\mathrm{Aut}(F_n)$ satisfies property $P$ with probability $p$,} if
\[\frac{\#\{w \in \mathcal{W}_{l,n} \mid w \text{ satisfies }P\}}{\#\mathcal{W}_{l,n}} \to p \text{ as }l \to \infty.\]
We say that a \emph{generic element in $\mathrm{Aut}(F_n)$ has property $P$}, if a random element satisfies property $P$ with probability $p = 1$.

An automorphism $\phi \in \mathrm{Aut}(F_n)$ is said to be \emph{super irreducible} if no positive power of the induced map $\phi_{\mathrm{ab}} \in \mathrm{GL}(n, \QQ)$ maps a proper subspace of $H_1(F_n;\QQ)$ into itself.  A free-by-cyclic group $G$ is \emph{super irreducible} if there exists some splitting $G \cong F_n \rtimes_{\phi} \Z$ such that $\phi$ is super irreducible.

The following theorem is a consequence of the results in Section~7 of \cite{Rivin2008}, which hold verbatim after replacing $\mathrm{SL}(n, \Z)$ by $\mathrm{GL}(n, \mathbb{Z})$ in all the statements.

\begin{thm}~\cite{Rivin2008}\label{Rivin}
A generic element in $\mathrm{Aut}(F_n)$ is super irreducible.
\end{thm}

\begin{prop}\label{genericBetti}
    For a generic element in $\mathrm{Aut}(F_n)$, the first Betti number of its mapping torus is equal to one.
\end{prop}

\begin{proof}Write $\phi_{\mathrm{ab}}$ to denote the image of $\phi$ under the natural map induced by the action on the abelianisation of $F_n$,
\[ \begin{split} \mathrm{Aut}(F_n) &\to \mathrm{GL}(n, \Z)  \\ \phi &\mapsto \Phi_{\mathrm{ab}}.\end{split}\]
    The free abelianisation of $F_n \rtimes_{\phi} \Z$ is isomorphic to $\Z$ if and only if $\phi_{\mathrm{ab}}$ has no eigenvalue equal to 1 \cite[Theorem~2.4]{BogopolskiMartinoVentura2007}. By \Cref{Rivin}, for a generic element in $\mathrm{Aut}(F_n)$ which represents the automorphism $\phi$, the characteristic polynomial of $\phi_{\mathrm{ab}}$ is irreducible over $\QQ$. Hence the result follows. 
\end{proof}

Write $\mathcal{H}_{l,n}$ to denote the set of free-by-cyclic presentations 
\[\mathcal{P} = \langle x_1, \ldots, x_n, t \mid t^{-1}x_it = \Phi(x_i), 1 \leq i \leq n \rangle \text{ for all }\Phi \in \mathcal{W}_{l,n}.\] We say that a \emph{generic $F_n$-by-cyclic group satisfies property $P$ with probability $p$}, if 
\[\frac{\# \{G \in \mathcal{H}_{l,n} \mid G \text{ satisfies }P\}}{\# \mathcal{H}_{l,n}} \to 1 \text{ as }l \to \infty.\]

\begin{prop}\label{generic_properties} A generic $F_n$-by-cyclic group has first Betti number equal to one and is super irreducible. 
\end{prop}

\subsection{Nielsen fixed point theory}\label{sec:Nielson}

Let $X$ be a connected compact polyhedral complex and $f \colon X \to X$ a continuous self-map. An \emph{m-periodic point} $p \in X$ is a fixed point of the map $f^m$. A path $\gamma$ between two $m$-periodic points $x$ and $y$ is an \emph{m-periodic Nielsen path} if $f^m(\gamma)$ is homotopic to $\gamma$. An indivisible $m$-periodic Nielsen path is such that $\gamma$ cannot be expressed as the concatenation $\gamma = \alpha \cdot \beta$, where $\alpha$ and $\beta$ are non-trivial $m$-periodic Nielsen paths. We will call a 1-periodic Nielsen path simply a \emph{Nielsen path}.

We can define an equivalence relation on the set of $m$-periodic points so that $x \sim y$ if there exists an $m$-periodic Nielsen path from $x$ to $y$. We call the equivalence classes under this relation \emph{m-point classes}. Each $m$-point class forms a so-called \emph{isolated subset} of $\Fix(f^m)$ and thus it is possible to define its \emph{index} (see \cite[Section~1.3]{Jiang1996} for the definition of fixed point index).

The map $f$ acts on the set of $m$-point classes and the action preserves index. We let $\mathrm{Orb}_m(f)$ be the set of orbits of $m$-periodic classes under this action. Each orbit $\mathcal{O} \in \mathrm{Orb}_m(f)$ determines a free homotopy class of loops in the mapping torus $M_f$, and thus a conjugacy class in $\pi_1(M_f)$, which we denote by $\mathrm{cd}(\mathcal{O})$. Furthermore, every $\mathcal{O} \in \mathrm{Orb}_m(f)$ admits an index $\mathrm{ind}_m(f; \mathcal{O}) \in \mathbb{Z}$, defined to be the index of any $m$-point class in the orbit. An $m$-periodic orbit $\mathcal{O} \in \mathrm{Orb}_m(f)$ is said to be \emph{essential} if it has non-zero index. 

\begin{defn} The \emph{$m$-th Nielsen number} of $f$, denoted by $N_m(f)$, is the number of essential $m$-periodic orbits $\mathcal{O} \in \mathrm{Orb}_m(f)$. 
\end{defn}

It is a standard fact from Nielsen fixed point theory (see e.g. \cite[Chapter 1]{Jiang1983} and \cite{Jiang1996}), that each Nielsen number is independent of the choice of topological representative of $\Phi \in \mathrm{Out}(\pi_1(X))$. Hence, we may write $N_{\infty}(\Phi)$ to denote 
\[ N_{\infty}(\Phi) = \mathrm{lim}\,\mathrm{sup}_{m\to \infty}N_m(f)^{1/m},\]
where $(f, \Gamma)$ is any topological representative of $\Phi$.

\begin{lemma}\label{lemma:Nielsen-counting}
    If $(f,\Gamma)$ is an improved relative train track then there exists a positive constant $K$ such that \[  F(m) - K \leq  N_m(f) \leq F(m) + K,\]
    where $F(m)$ is the number of $f^m$-fixed points in the interior of exponentially growing edges.
\end{lemma}

\begin{proof}
Fix a maximal $f$-invariant filtration of $\Gamma$, 
\[\emptyset = \Gamma_0 \subseteq \Gamma_1 \subseteq \ldots \subseteq \Gamma_l = \Gamma,\]
with $S_i = \mathrm{Cl}(\Gamma_i \setminus \Gamma_{i-1})$ for each $1 \leq i \leq l$.

Suppose first that $e \in E(\Gamma)$ is a polynomially-growing edge. Then there exists a polynomially-growing stratum $S_i$ such that $S_i = \{e\}$ and the edge $e$ is either fixed by $f$, or $f(e) = e\gamma$ where $\gamma$ is an immersed loop in $\Gamma_{i-1}$. Hence $f^m(e) = e \gamma'$ for some loop $\gamma' \in \Gamma_{i-1}$. Thus, the interior of $e$ contributes at most one $f^m$-fixed point class. Moreover, each vertex of $\Gamma$ contributes at most one $f^m$-fixed class. Hence, $N_f(m) \leq F(m) + K_1$, where $K_1$ is the number polynomially-growing edges plus the number of vertices in $\Gamma$.

We now consider fixed points contained in the interiors of exponentially-growing edges. By \cite[Theorem~5.1.5]{BestvinaFeighnHandel2000}, each periodic Nielsen path in $\Gamma$ has period one. Moreover, for each exponentially-growing stratum $S_i$, there exists at most one indivisible Nielsen path $\gamma$ that intersects $S_i$ non-trivially, and the initial (partial) edges of $\gamma$ and $\gamma^{-1}$ are contained in $S_i$. Also, it is clear that all the fixed point classes contained in the interior of exponentially-growing strata are essential. It follows that $N_m(f)$ is bounded below by $F(m)$ minus the number of exponentially-growing strata in $\Gamma$, which we denote by $K_2$. Hence, for every $m \in \NN$,
\[F(m) - K_2\leq  N_m(f) \leq F(m) + K_1.\] The result follows by setting $K = \mathrm{max}\{K_1, K_2\}$.


 \end{proof}

\begin{prop}\label{train_track}  Let $\Phi \in \mathrm{Out}(F_n)$ be an outer automorphism with stretch factor $\lambda > 1$. Then $N_{\infty}(\Phi)$ is equal to $\lambda$.
\end{prop}

\begin{proof}

By \cite{BestvinaFeighnHandel2000}, there exists a positive integer $k$ such that $\Phi^k$ admits an improved relative train track representative $(f, \Gamma)$. Let $\lambda$ be the stretch factor of $\Phi^k$. We will start by proving that 
\[N_{\infty}(\Phi^k) = \lambda.\]

Let $A$ be the incidence matrix corresponding to $(f, \Gamma)$ and fix a maximal $f$-invariant filtration of $\Gamma$. Let $\{S_i\}_{i\in I}$ be the set of exponentially-growing strata of $\Gamma$ and let $\lambda_i$ be the stretch factor of $S_i$. Let $A_i$ denote the sub-matrix of $A$ corresponding to the edges in $S_i$. 

By \cref{lemma:Nielsen-counting}, there exists a constant $K$ such that for every $m \in \NN$, \begin{equation}\label{eq:Nielsen-counting} F(m) - K \leq  N_m(f) \leq F(m) + K,\end{equation} where $F(m)$ is the number of $f^m$-fixed points in the interior of the exponentially-growing edges.

Fix an exponentially-growing edge $e$ in $\Gamma.$ The number of fixed points of $f^m$ contained in the interior of $e$ is exactly the number of times the edge-path $f^m(e)$ crosses the edge $e$ in either direction minus a constant $C_{m,e} \in \{0,1,2\}$. Indeed, if the edge-path $f^m(e)$ traverse at least two edges and begins or ends with the edge $e$, then the fixed point corresponding to $e$ will not be in the interior of $e$. Note also that the number of times $f^m(e)$ crosses the edge in either direction is given by the element on the diagonal of $A^m$ corresponding to $e$.

Combining the argument in the previous paragraph with \eqref{eq:Nielsen-counting} we obtain that
\[N_m(f) \asymp \sum_{i\in I} \mathrm{tr}(A_i^m),\]
where for any two functions $g_1,g_2 \colon \NN \to \RR$, we write $g_1 \asymp g_2$ if there exists a constant $K > 0$ such that for all $m \in \mathbb{N}$,
\[ g_2(m) - K \leq g_1(m) \leq g_2(m) + K.\]

For each matrix $A_i$, let $n_i$ denote the order of $A_i$ and let $\lambda_{i,j}$ be its eigenvalues, for $1 \leq j \leq n_i$. Then 
\[\mathrm{tr}(A_i^m) = \sum_{1\leq j \leq n_i} \lambda_{i,j}^m.\]
Since $\lambda$ is the maximal stretch factor, have that $\left| \lambda_{i,j} / \lambda \right|^m \leq 1$ for each $i\in I$ and $j \leq n_i$, with equality for some $i,j$. Hence,
\[ \begin{split}\mathrm{lim}\,\mathrm{sup}_{m\to \infty}N_m(f)^{1/m} &= \lambda \cdot \lim \sup_{m\to \infty} \left(\sum_{i\in I}\sum_{1\leq j \leq n_i}(\lambda_{i,j}/\lambda)^m \right)^{1/m}  \\
&= \lambda.\end{split}\]
Thus, $N_{\infty}(\Phi^k)$ is equal to the stretch factor of $\Phi^k$.

By \cite[Corollary~7.14]{FrancavigliaMartino2021}, if $\lambda$ is the maximal stretch factor of a relative train track representative of $\Phi^k$, then $\lambda^{1/k}$ is the maximal stretch factor associated to $\Phi$. Note also that $N_{\infty}(\Phi^k) = N_{\infty}(\Phi)^k$. The result follows by combining the arguments in the previous paragraphs. \end{proof}

\subsection{Detecting atoroidal monodromy}\label{sec:hyp}
In this section we will prove that hyperbolicity (equivalently the property of admitting atoroidal monodromy) is determined by the profinite completion.  The strategy is to show finitely generated abelian subgroups of free-by-cyclic groups are fully separable and then use work of Brinkmann \cite{Brinkmann2000} and Wilton--Zalesskii \cite{WiltonZalesskii2017}.

Recall a subgroup $H\leqslant G$ is \emph{separable} if for every $g\in G\backslash H$ there exists a finite quotient $\rho\colon G\onto Q$ such that $\rho(g)\not\in\rho(H)$.  A subgroup $H$ is \emph{fully separable} in $G$, if every finite index subgroup of $H$ is separable in $G$. We will need the following lemma:

\begin{lemma}[{\cite[Lemma~4.6]{Reid2015}}]\label{Reid lemma}
    Let $G$ and $H \leq G$ be finitely generated. If $H$ is fully separable in $G$ then the closure of $H$ in $\widehat{G}$ is isomorphic to $\widehat{H}$.
\end{lemma}

\begin{prop}\label{sep reducible}
Let $G$ be a free-by-cyclic group and let $H\leqslant G$ be a finitely generated subgroup. If $H \leq G$ is free-by-cyclic or abelian then $H$ is separable in $G$. 
\end{prop}

\begin{proof}
Fix a fibred character $\varphi \colon G \to \mathbb{Z}$ of $G$. Let $F = \mathrm{ker}\varphi$ be the fibre, $t \in \varphi^{-1}(1)$ and let $\phi \in \mathrm{Aut}(F)$ be the automorphism corresponding to the conjugation action of $t$ on $F$ in $G$. Fix $H \leq G$ a free-by-cyclic subgroup and $g \in G \setminus H$.

Suppose first that $H$ is not contained in $F$. By \cite[Proposition~2.3]{FeighnHandel1999}, there exist a finitely generated subgroup $A \leq F$, an element $y\in F$ and a positive integer $k$ such that $\phi^k(A) = yAy^{-1}$, and 
\[H = \langle A, t^ky \rangle \simeq  A \rtimes \langle t^ky \rangle.\] 
Let $g \in G \setminus H$. Then $g = bt^m$, for some $b \in F$ and $m \in \mathbb{Z}$. Suppose that $m$ is not a multiple of $k$. Consider the finite quotient $\rho_k \colon G \to \Z / k\Z$ of $G$, which maps each element of $F$ to 0, and which sends $t$ to a generator of the cyclic group  $\Z / k\Z$ of order $k$. It follows that $\rho_k(H) = 0$ and $\rho_k(g)\neq 0$.

Suppose now that $m = kl$ for some $l \in \mathbb{Z}$. Then $g = b' (t^ky)^l$, for some $b' \in F$, and since $g \not \in H$ it must be that $b' \not \in A$. The usual Marshall--Hall argument gives a finite-index subgroup $F' \leq F$ such that $b' \not\in F'$ and $A \leq F'$. Let $N = [F : F']$. Let $\mathrm{ad}_y$ denote the inner automorphism of $F$ which acts by conjugation with $y$. Note that $\mathrm{ad}_y \cdot \phi^k \colon F \to F$ permutes the subgroups of $F$ of index $N$. Hence there exists some positive integer $M$ such that $(\mathrm{ad}_y \cdot \phi^k)^M(F') = F'$. Let $F'' = \bigcap_{i=0}^{M-1} (\mathrm{ad}_y \cdot \phi^k)^i(F')$. Then $\mathrm{ad}_y \cdot \phi^k (F'') = F''$ and $A \leq F''$. Furthermore, since $F'' \leq F'$, we have that $b' \not \in F''$. Thus $G' = \langle F'', t^ky \rangle \cong F'' \rtimes \langle t^ky \rangle$ is a finite index subgroup of $G$, such that $g \not \in G'$ and $H \leq G'$. 

Suppose now that $H$ is contained in $F$. Since $\ker \varphi$ is free, it follows that $H$ is infinite cyclic. Let $a \in F$ be a generator of $H$. Let $g = bt^m$ for some $b \in F$ and $m \in \Z$. If $m \neq 0$, consider the finite quotient $\rho_{m+1} \colon G \to \Z / (m+1)\Z$ which sends $F \to 0$ and $t$ to a generator of the cyclic group $\Z /(m+1)\Z$. Then $\rho_{m+1}(H) = 0$ and $\rho(g) \neq 0$.

Suppose finally that $g \in F$, Since $g \not \in \langle a \rangle \leq F$, by Marshall--Hall's theorem, there exists a finite index subgroup $F' \leq F$ such that $\langle a \rangle \leq F'$ and $b \not \in F'$. Moreover, there exists a positive integer $M$ such that $\phi^M(F') = F'$ since $\phi$ permutes finite index subgroups of a given index in $F$. Hence, 
\[G' := \langle F', t^M \rangle \cong F' \rtimes \langle t^M \rangle.\]
We have that $\langle a \rangle \leq F' \leq G'$ and since $b \not \in F'$, it must be the case that $b \not \in G'$. This completes the proof. \end{proof}




\begin{corollary}\label{fully_separable}
Let $G$ be a free-by-cyclic group.  If $H\leq G$ is a free-by-cyclic or abelian subgroup, then $H$ is fully separable in $G$.  In particular $\bar H$, the closure of $H$ in $\widehat{G}$, is isomorphic to $\widehat{H}$.
\end{corollary}
\begin{proof}
Every finite-index subgroup of $H$ is free-by-cyclic or abelian.  It follows from \Cref{sep reducible} that every finite-index subgroup of $H$ is separable in $G$. The final part follows by \cref{Reid lemma}.
\end{proof}

We have everything we need to prove \Cref{thmx:hyperbolicity} from the introduction.

\medskip
\setcounter{thmx}{2}
\begin{thmx} \label{thmx:hyperbolicity}
Let $G_A$ and $G_B$ be profinitely isomorphic free-by-cyclic groups. Then $G_A$ is Gromov hyperbolic if and only if $G_B$ is Gromov hyperbolic.
\end{thmx}

\smallskip

\begin{proof}
Let $G_A$ and $G_B$ be free-by-cyclic groups such that $\widehat G_A\cong \widehat G_B$. Suppose that $G_A$ is Gromov hyperbolic. By \cite{HagenWise2015}, $G_A$ is a cocompactly cubulated and thus virtually special.  Hence we may apply \cite[Theorem~D]{WiltonZalesskii2017} to deduce that $\widehat\ZZ^2$ is not a subgroup of $\widehat G_A$.  By \Cref{fully_separable}, the $\Z^2$ subgroups of $G_B$ are fully separable and since $\widehat G_B$ contains no $\widehat\ZZ^2$ subgroups, it follows $G_B$ contains no $\ZZ^2$ subgroups.  In particular, by \cite[Theorem~1.2]{Brinkmann2000} $G_B$ is Gromov hyperbolic.

Suppose conversely that $G_A$ is not Gromov hyperbolic. Then by \cite{Brinkmann2000}, $G_A$ has a $\ZZ^2$ subgroup and so by \Cref{fully_separable}, $\widehat{G}_A$ contains a $\widehat{\ZZ}^2$ subgroup.  Suppose now $G_B$ is not Gromov hyperbolic.  Then, by the argument in the previous paragraph, $\widehat{G}_B$ does not contain $\widehat{\ZZ}^2$ subgroups.  This contradiction completes the proof.
\end{proof}

We will need the following proposition later. It is proved in \cite[Lemma~2.2]{BridsonReid2020} but we include a proof for completeness.

\begin{prop}[{\cite[Lemma~2.2]{BridsonReid2020}}]\label{fully separable fibres}
Let $G$ be a group and $\varphi \colon G \to \Z$ an epimorphism. If $N = \ker\varphi$ is finitely generated, then $N$ is fully separable in $G$.
\end{prop}

\begin{proof}
We show that every finite index subgroup $H\leq_f N$ of $N$ is separable in $G$.

Pick an element $t \in \varphi^{-1}(1)$ and an automorphism $\phi \in \mathrm{Aut}(N)$ induced by the conjugation action of $t$. Let $g \in G \setminus H$. Then $g = bt^m$, for some $b \in N$ and $m \in \mathbb{Z}$. If $m \neq 0$, consider the finite quotient $\rho_{m+1} \colon G \to \Z / (m+1)\Z$ which sends $N \mapsto 0$ and $t$ to a generator of the cyclic group $\Z /(m+1)\Z$. Then $\rho_{m+1}(H) = 0$ and $\rho_{m+1}(g) \neq 0$. Suppose now that $g \in N$. Since $g \not \in H \leq N$, and since $H \leq_f N$ is finite index and thus separable, there exists a finite index subgroup $N' \leq N$ such that $H\leq N'$ and $b \not \in N'$. Moreover, since $\phi$ permutes finite index subgroups of a given index in $N$, there exists a positive integer $M$ such that $\phi^M(N') = N'$. Hence, 
\[G' := \langle N', t^M \rangle \cong N' \rtimes \langle t^M \rangle.\]
We have that $H\leq N' \leq G'$ and since $b \not \in N'$, it must be the case that $b \not \in G'$. This completes the proof.

\end{proof}

\section{Some properties of twisted Alexander polynomials and Reidemeister torsion}\label{sec:AP}
In this section we will collect a number of facts about twisted Alexander polynomials and twisted Reidemeister torsion that we will use later on.  For a survey on twisted invariants (in the context of $3$-manifolds) see \cite{FriedlVidussi2011survey}.  Our main contribution is a complete computation of the zeroth Alexander polynomials twisted by representations factoring through finite groups over characteristic zero fields (\Cref{zeroth AP palindromic}).

\begin{defn}[Alexander modules and polynomials]
Let $R$ be a unique factorisation domain and let $G$ be a finitely generated group.  Let $\varphi$ be a non-trivial primitive class in $H^1(G;\Z)$ considered as a homomorphism $G\onto \Z$ and let $\rho\colon G\to \GL_n(R)$ be a representation. Consider $R^n[t^{\pm 1}]$ equipped with the $RG$-bimodule structure given by 
\[g.x = t^{\varphi(g)}\rho(g)x, \quad x.g = xt^{\varphi(g)}\rho(g) \]
 for $g \in G, x \in  R^n[t^{\pm 1}]$.  For $n\in\ZZ$, we define the \emph{$k$th twisted Alexander module of $\varphi$ and $\rho$} to be $H_k(G;R^n[t^{\pm 1}])$, where $R^n[t^{\pm 1}]$ has the right $RG$-module structure described above. Observe that $H_k(G;R^n[t^{\pm 1}])$ is a left 
 $R[t^{\pm 1}]$-module.
If $G$ is of type $\mathsf{FP}_k(R)$, then the $k$th twisted Alexander module is a finitely generated $R[t^{\pm 1}]$-module.  Moreover, it is zero whenever $k<0$ or $k$ is greater than the cohomological dimension of $G$ over $R$.

Since $R$ is a UFD so is $R[t^{\pm1}]$.  Let $M$ be an $R[t^{\pm1}]$-module.  The \emph{order} of $M$ is the greatest common divisor of all maximal minors in a presentation matrix of $M$ with finitely many columns.  The order of $M$ is well-defined up to a unit of $R[t^{\pm1}]$ and depends only on the isomorphism type of $M$. 

Suppose that $G$ is of type $\mathsf{FP}_k(R)$.  The \emph{$k$th twisted Alexander polynomial} $\Delta_{k,R}^{\varphi,\rho}(t)$ over $R$ with respect to $\varphi$ and $\rho$ is defined to be the order of the $k$th twisted (homological) Alexander module of $\varphi$ and $\rho$, treated as a left $R[t^{\pm 1}]$-module.
\end{defn}

We will now collect a number of facts about twisted Alexander polynomials. Let $R$ be a unique factorisation domain. Given any polynomial $c(t)\in R[t^{\pm1}]$ where $c(t)=\sum_{i=0}^r c_it^i$ we write $c^\bigstar(t)$ for the polynomial $\sum_{i=0}^r c_{r-i}t^i$.  For $p(t),q(t)\in R[t^{\pm1}]$ we write $p(t)\doteq q(t)$ if $p(t)=uq(t)$ where $u\in R[t^{\pm 1}]$ is a unit. The following lemma is a triviality.

\begin{lemma}\label{AP conjugacy}
    Let $G$ be a group of type $\mathsf{FP}_n(R)$, let $\varphi\colon G\twoheadrightarrow\ZZ$, and let $\rho,\sigma\colon G\to\GL_n(R)$ be representations of $G$ over a UFD $R$.  If $\rho$ and $\sigma$ are conjugate representations, then
\[\Delta_n^{\varphi,\rho}(t)\doteq \Delta_n^{\varphi,\sigma}(t). \]
\end{lemma}

\begin{lemma}\label{AP direct sum formula}
Let $G$ be a group of type $\mathsf{FP}_n(R)$, let $\varphi\colon G\twoheadrightarrow\ZZ$, and let $\rho,\sigma\colon G\to\GL_n(R)$ be representations of $G$ over a UFD $R$.  Then,
\[\Delta_n^{\varphi,\rho\oplus\sigma}(t)\doteq \Delta_n^{\varphi,\rho}(t)\times\Delta_n^{\varphi,\sigma}(t). \]
\end{lemma}
\begin{proof}
This follows from the fact that homology commutes with taking direct sums of coefficient modules. 
\end{proof}

The following lemma is a triviality

\begin{lemma}\label{AP -phi recip}
Let $R$ be a UFD.  Let $G$ be a group, let $\varphi\colon G\twoheadrightarrow\ZZ$, and let $\rho\colon G\to\GL_k(R)$ be a representation. Then,
\[(\Delta^{\varphi,\rho}_{n})^\bigstar(t)\doteq\Delta^{-\varphi,\rho}_{n}(t)\doteq\Delta^{\varphi,\rho}_{n}(t^{-1})  \]
up to monomial factors with coefficients in $R^\times$.
\end{lemma}

The next lemma will be a key step in proving profinite rigidity of twisted Reidemeister torsion for our class of free-by-cyclic groups.  For a $G$-module $M$ being acted on via $\alpha\colon G\times M\to M$ we write $M_\alpha$ when we wish to make clear the $G$-module structure.

\begin{lemma}\label{zeroth AP palindromic}
Let $G$ be a finitely generated group, let $\varphi\colon G\twoheadrightarrow\ZZ$ be algebraically fibred, and let $\rho\colon G\twoheadrightarrow Q\to\GL_k(\QQ)$ be a representation factoring through a finite group. Then,
\[\Delta^{\varphi,\rho}_0(t)\doteq (1-t)^m P(t),\] 
where $m\geq 0$ and $P(t)$ is a product of cyclotomic polynomials, up to multiplication by monomials with coefficients in $\QQ^\times$.  In particular,
\[\Delta^{\varphi,\rho}_{0}(t)\doteq\Delta^{\varphi,\rho}_{0}(t^{-1}). \]
\end{lemma}
\begin{proof}
Let $F$ denote the kernel of $\varphi$.  We need to compute $M\coloneq H_0(G;\QQ^k[t^{\pm1}])$ which is naturally isomorphic to the coinvariants $(\QQ^k[t^{\pm1}])_G$.  

By Maschke's Theorem we may write the representation $\rho$ of $Q$ as a sum $\oplus_{i=1}^\ell\rho_i\colon Q\to \prod_{i=1}^\ell \GL_{k_i}(\QQ)$, where $\sum_{i=1}^\ell k_i =k$, of irreducible $\QQ$-representations of $L$.   We may now write \[M = \bigoplus_{i=1}^\ell (\QQ^{k_i}[t^{\pm1}])_G.\]

For each $i$ there are three possibilities:

\paragraph{\underline{\textbf{Case 1:}}} $\rho_i(Q)\neq \{1\}$ but $\rho_i(F)=\{1\}$.

In this case $\rho_i$ has image a non-trivial finite cyclic group $L$.  We quickly recap the $\QQ$-representation theory of $\Z/n$ for $n\geq 2$.  Recall that $\QQ[\Z/n]=\QQ[X]/(X^n-1)$ so the irreducible representations of $\Z/n$ are exactly the cyclotomic fields $\QQ(\chi_d)$ for each $d$ dividing $n$.  These representations are exactly given by the quotient map $\pi_d\colon\QQ[\Z/n]\to \QQ(\chi_d)$.
Note that in this case for a generator $g$ of $\Z/n$ the characteristic polynomial of $\pi_d(g)$ is the cyclotomic polynomial $\chi_d$.

Since $\rho_i$ is irreducible it follows that we are in the situation of a cyclotomic representation.  Consider the tail end of the standard resolution for $\Z$ over $\Z G$
\[\begin{tikzcd}
C_1 \arrow[r,"\partial"] & C_0 \quad =\quad a_0\Z G\oplus\dots\oplus a_{m-1}\Z G \oplus t\Z G \arrow[r,"\partial"] & \Z G \end{tikzcd}\]
where $a_0,\dots,a_{m-1}$ is a generating set for $F$, where $t$ is the generator of $\Z$ viewing $G=F\rtimes\Z$, and where 
\begin{equation}\label{eqn.partial0}
\partial = \begin{bmatrix} 1-a_0,\dots,1-a_{m-1},1-t \end{bmatrix}.
\end{equation}

We need to compute the order of the presentation matrix 
\[
\partial\otimes_{\Z G}\id_{\QQ[\Phi_d][t^{\pm1}]}=\begin{bmatrix} 0,\dots,0,\id-\rho_i(t)t \end{bmatrix}.
\]
But this is the same as computing an order of the square matrix $\id-\rho_i(t)t$.  
Now,
\begin{equation}\label{eqn.ordrhot}
    \ord(\id-\rho_i(t)t)\doteq \det(\id t^{-1}-\rho_i(t)t\cdot t^{-1})t^{p-1} \doteq \det(\id t^{-1}-\rho_i(t))
\end{equation}
but this is exactly the characteristic polynomial of $\rho_i(t)$ with respect to $t^{-1}$.  Namely, it is the cyclotomic polynomial $\chi_d(t^{-1})$ but this is palindromic of even degree, $t-1$, or $t+1$ so we have that $\Delta_0^{\varphi,\rho_i}(t)\doteq\chi_d(t)$. \hfill$\blackdiamond$

\paragraph{\underline{\textbf{Case 2:}}} $\rho_i(F)\neq\{1\}$.

We start by again by viewing $G$ as $F\rtimes \Z$.  In particular, we have a differential $\partial$ as in \eqref{eqn.partial0} such that $\Delta_0^{\varphi,\rho_i}$ is given by an order of 
\[D_i\coloneqq \partial\otimes_{\Z G}\id_{\QQ^{k_i}[t^{\pm1}]}=[\id - \rho_i(a_0),\dots,\id-\rho_i(a_{m-1}),\id-\rho_i(t)].\]
To this end we define $D$ to be the set of cofactors of $D_i$.  So  $\Delta_0^{\varphi,\rho_i}\doteq \gcd D$.

We first conjugate $\rho_i$ so that $\rho_i(t)$ is in block diagonal form.  Since the image of $t$ is cyclic, say of order $n$, we obtain a block structure where the non-identity blocks are matrices corresponding to non-trivial $\QQ$-representations of various subgroups $H\leqslant \Z/n$.  Thus, arguing as in \eqref{eqn.ordrhot} we see that \[(1-t)^{n'}\cdot \prod_{j=1}^\ell\chi_{n_j}(t)\in D,\] where $n'$ is dimension of the fixed subspace of $\rho_i(t)$ and $\chi_{n_j}(t)$ is the cyclotomic polynomial of order ${n_j}$ such that $n_j$ divides $n$.

Now, $\Delta_0^{\varphi,\rho_i}$ divides every element of $D$ and is a polynomial defined over $\QQ[t]$ (up to multiplication by $t^\ell$ for some $\ell\geq 0$) and $\chi_{n_j}(t)$ is the minimal polynomial for all primitive $n_j$th roots of unity.
In particular, any non-trivial polynomial dividing and not equal to $\chi_{n_j}(t)$ is not defined over $\QQ[t^{\pm1}]$.  It follows that $\Delta_0^{\varphi,\rho_i}=P_i(t)\cdot(1-t)^{n''}$ where $P_i(t)$ is a product of cyclotomic polynomials and $n''$ is a non-negative integer less than or equal to $k_i$. \hfill$\blackdiamond$


\paragraph{\underline{\textbf{Case 3:}}} $\rho_i(G)=\{1\}$.

In this case we are computing $(\QQ[t^{\pm1}])_G$ where $G$ acts trivially on $\QQ$.  Clearly, this is isomorphic to $\QQ[t^{\pm1}]/(1-t)$ which is additively isomorphic to $\QQ$.\hfill $\blackdiamond$

By \Cref{AP direct sum formula} we have that $\Delta^{\varphi,\rho}_0(t)\doteq \prod_{i=1}^\ell \Delta^{\varphi,\rho_i}_0(t)\doteq (1-t)^nP(t)$ where $n$ is some non-negative integer and $P(t)$ is a product of cyclotomic polynomials.  

The ``in particular'' now follows from the fact cyclotomic polynomials are palindromic (provided $d\neq2$) or equal to $t-1$ and an easy computation:  Write $P(t)=(t-1)^{m'}P'(t)$ where $m'$ is the number of $(t-1)$ factors of $P(t)$.  Let $\delta$ denote the degree of $P(t)$, let $\epsilon=1$ if exactly one of $m$ and $m'$ are non-zero, and let $\epsilon=0$ otherwise. Now, 
\begin{align*}
\Delta_0^{\varphi,\rho}(t^{-1})&\doteq (-1)^\epsilon t^{m+m'+\delta}(1-t^{-1})^m(t^{-1}-1)^{m'}P'(t^{-1})\\
&\doteq (1-t)^m(t-1)^{m'}P'(t)\\
&\doteq \Delta_0^{\varphi,\rho}(t). \qedhere
\end{align*}
\end{proof}

\begin{remark}
The previous lemma easily generalises to any field $\FF$ of characteristic  zero with the modified conclusion that $\Delta_0^{\varphi,\rho}(t)\doteq Q(t)P(t)$, where $Q(t)$ is a product of polynomials $(1-\zeta_i t)$ such that $\zeta_i$ is some root of unity in $\FF$, and where $P(t)$ is a product of cyclotomic polynomials whose roots do not lie in $\FF$.
\end{remark}

Let $R$ be a unique factorisation domain.  A polynomial $c(t)\in R[t^{\pm1}]$ is \emph{palindromic} if $c(t)=\sum_{i=0}^r c_it^i$ and $c_i=c_{r-i}$.  Given any polynomial $c(t)\in R[t^{\pm1}]$ where $c(t)=\sum_{i=0}^r c_it^i$ recall that we write $c^\bigstar(t)$ for the polynomial $\sum_{i=0}^r c_{r-i}t^i$.  Note that $c(t)\cdot c^\bigstar(t)$ is palindromic.

For a group $G$ we let $\mathbf{1}$ denote the trivial homomorphism $G\onto\{1\}$.

The following lemma is well known to experts.  We include a proof for completeness.

\begin{lemma}\label{ordAP is Bno}
Let $\FF$ be a field.  Let $G$ be a group of type $\mathsf{FP}_n(\FF)$.  If $\varphi\colon G\to \Z$ is an $\mathsf{FP}_n(\FF)$-fibring, then 
\[
\deg \Delta_{G,n}^{\varphi,\mathbf{1}}(t)= b_n(\ker\varphi;\FF),
\]
where the Alexander polynomial is taken over $\FF$.
\end{lemma}
\begin{proof}
We may write $G$ as $\ker\varphi\rtimes\langle t\rangle$ and $\Delta_{G,n}^{\varphi,\mathbf{1}}(t)$ as the characteristic polynomial of the $\FF$-linear transformation $T_n\colon H_n(\ker\varphi;\FF)\to H_n(\ker\varphi;\FF)$ and $T^n\colon H^n(\ker\varphi;\FF)\to H^n(\ker\varphi;\FF)$, where $T$ is the induced map of $t$ on (co)homology.  Hence, \[H^n(\ker\varphi;\FF)\cong \FF[t^{\pm1}]/(\Delta_{G,n}^{\varphi,\mathbf{1}}(t)).\qedhere\]
\end{proof}


\begin{lemma}\label{top dim AP}
Let $R$ be a UFD.  Let $G$ be a group of type $\mathsf{F}$ admitting a compact $K(G,1)$ of dimension $n$, let $\varphi\colon G\twoheadrightarrow\ZZ$, and let $\rho\colon G\to\GL_k(R)$ be a representation.  If $\Delta^{\varphi,\rho}_{n}\neq 0$ over $R$, then $\Delta^{\varphi,\rho}_{n}\doteq 1$.
\end{lemma}
\begin{proof}
Consider the head end of the cellular chain complex for $G$, namely,
\[\begin{tikzcd}
0 \arrow[r] & C_n \arrow[r,"\partial_{n-1}"] & C_{n-1} \arrow[r] & \cdots
\end{tikzcd}\]
tensoring with $R^k[t^{\pm1}]$ and taking homology we see that
$H_n(G;R^k[t^{\pm1}])=\ker \partial_{n-1}\otimes \id_{R^k[t^{\pm1}]}$.  In particular, it is a submodule of a free $R[t^{\pm1}]$-module and so cannot be $R[t^{\pm1}]$-torsion unless it is $0$.  But since $\Delta^{\varphi,\rho}_{n}\neq 0$ by assumption, we have that $H_n(G;R^k[t^{\pm1}])$ is $R[t^{\pm1}]$-torsion.  The result follows.
\end{proof}

We now wish to define \emph{the twisted Reidemeister torsion} $\tau^{\varphi,\rho}_{G,R}(t)$ of $\varphi$ twisted by $\rho$ over $R$.  Rather than give the original definition which we will not need, we instead use the following lemma which recasts the invariant in terms of twisted Alexander polynomials as our definition.  The lemma can be deduced by standard methods, for example, it is an immediate corollary of \cite[Lemma 2.1.1]{Turaev1986}.

\begin{lemma}\label{defn RT}
Let $R$ be a UFD.  Let $G$ be a group of type $\mathsf{F}$, let $\varphi\colon G\twoheadrightarrow\ZZ$ have kernel of type $\mathsf{F}$, and let $\rho\colon G\to\GL_k(R)$ be a representation.  Then,
\[\tau^{\varphi,\rho}_{G,R}(t)\doteq \prod_{n\geq 0}\left(\Delta_{G,n}^{\varphi,\rho}(t)\right)^{(-1)^{n+1}} \]
up to monomial factors with coefficients in $\mathrm{Frac}(R)^\times$.
\end{lemma}

This allows us to easily compute the Reidemeister torsion of free-by-cyclic groups.

\begin{prop}\label{defn RT FbyZ}
Let $R$ be a UFD.  Let $G=F_n\rtimes_\varphi\Z$ be a free-by-cyclic group and let $\rho\colon G\to\GL_k(R)$ be a representation.  Then,
\[\tau^{\varphi,\rho}_{G,R}(t)=\frac{\Delta_{G,1}^{\varphi,\rho}(t)}{\Delta_{G,0}^{\varphi,\rho}(t)} \]
up to monomial factors with coefficients in $\mathrm{Frac}(R)^\times$.
\end{prop}
\begin{proof}
This follows from \Cref{top dim AP} and \Cref{defn RT}.
\end{proof}

The final well known lemma is elementary.

\begin{lemma}\label{conjugacy invariance of RT}
Let $R$ be a UFD.  Let $G$ be a group of type $\mathsf{F}$ admitting a character $\varphi\colon G\twoheadrightarrow\Z$ which has kernel of type $\mathsf{F}$.  If $\rho_1$ and $\rho_2$ are conjugate representations of $G$ into $\GL_k(R)$, then $\tau^{\varphi,\rho_1}_{G,R}(t)\doteq\tau^{\varphi,\rho_2}_{G,R}(t)$.
\end{lemma}

\section{Regularity}\label{sec:regularity}
In this section we will introduce the definition of a $\widehat{\Z}$-regular isomorphism.  We will prove that in the case where $G$ has $b_1(G)=1$ every profinite isomorphism is $\widehat{\Z}$-regular and deduce some consequences.

\begin{defn}[Corresponding quotients]\label{corresponding quotients}
Let $G_A$ and $G_B$ be residually finite groups. Suppose there exists an isomorphism $\Theta \colon \widehat{G}_A \to \widehat{G}_B$. Let $Q$ be a finite group. A pair of quotient maps $\gamma_A \colon G_A \twoheadrightarrow Q$ and $\gamma_B\colon G_B \to Q$ is said to be \emph{$\Theta$-corresponding}, if $\gamma_A$ is given by the composite
\begin{equation}
    \begin{tikzcd}
        G_A \arrow[r, "i"] & \widehat{G}_A \arrow[r, "\Theta"] & \widehat{G}_B \arrow[r, "\widehat{\gamma}_B"] & Q
    \end{tikzcd}
\end{equation}
Here, $i \colon G_A \to \widehat{G}_A$ denotes the natural inclusion and $\widehat{\gamma_B}$ denotes the (profinite) completion of $\gamma_B$.
\end{defn}

\begin{defn}[Matrix coefficient modules]
Let $H_A$ and $H_B$ be a pair of finitely generated $\ZZ$-modules.  Let $\Theta \colon \widehat{H_A}\to\widehat{H_B}$ be a continuous homomorphism of the profinite completions.  We define the \emph{matrix coefficient module} 
\[\MC(\Theta;H_A,H_B)\]
 (or simply 
$\MC(\Theta)$ if there is no chance of confusion) for $\Theta$ with respect to $H_A$ and $H_B$ to be the smallest $\Z$-submodule $L$ of $\widehat{\Z}$ such that $\Theta(H_A)$ lies in the submodule $H_B\otimes_\Z L$ of $\widehat{H_B}$.  We denote by 
\[\Theta^\MC\colon H_A\to H_B\otimes_\ZZ\MC(\Theta)\]
 the homomorphism uniquely determined by the restriction of $\Theta$ to $H_A$.

For a finitely generated group $G$ let $G^{\mathrm{fab}}$ denote the free part of the abelianisation $G^{\mathrm{ab}}$.  That is, the quotient of the abelianisation of $G$ by its torsion elements.

Given groups $G_A$ and $G_B$ and a continuous homomorphism $\Theta\colon \widehat{G}_A\to\widehat{G}_B$, we have an induced continuous homomorphism $\Theta_\ast\colon\widehat{G}_A^{\mathrm{fab}}\to\widehat{G}_B^{\mathrm{fab}}$.
We define $\MC(\Theta)\coloneqq\MC(\Theta_\ast,G_A^{\mathrm{fab}},G_B^{\mathrm{fab}})$.
\end{defn}

\begin{defn}[$\widehat \Z$-regular isomorphism] \label{defn:Zhatreg}
The isomorphism $\Theta \colon \widehat{G}_A \to \widehat{G}_B$ is \emph{$\widehat \Z$-regular}, if there exists a unit $\mu \in \widehat{\Z}^{\times}$ and 
an isomorphism $\Xi\colon G_A^{\mathrm{fab}} \to G_B^{\mathrm{fab}}$ such that $\Theta_{*}$ is the profinite completion of the map given by the composite
\begin{equation}\label{composite}
\begin{tikzcd}
G_A^{\mathrm{fab}} \arrow[r, "\Xi"] & G_B^{\mathrm{fab}} \arrow[r, "\cdot\times\mu"] &\widehat{G_B^{\mathrm{fab}}}.
\end{tikzcd}
\end{equation}
We sometimes write $\Theta_{*}^{1/\mu} \colon G_A^{\mathrm{fab}} \to G_B^{\mathrm{fab}}$ to denote the map $\Xi$ in \eqref{composite} and $\Theta^{*}_{1/\mu} \colon H^1(G_B, \Z) \to H^1(G_A, \Z)$ to denote its dual. 

For any $\varphi \in H^1(G_B; \Z)$ and $\psi \in H^1(G_A; \Z)$, we say $\psi$ is the \emph{pullback of $\varphi$ via $\Theta$}, if $\psi = \Theta^{*}_{1/\mu}(\varphi)$.  

We extend this definition to finite index subgroups as it will be useful later on. 
 Suppose $L_A$ is a finite index normal subgroup of $G_A$ and let $L_B$ be the corresponding normal subgroup of $G_B$ under $\Theta$.  If $\psi \in H^1(G_A; \Z)$ is the pullback of $\varphi$ via $\Theta$, then we say \emph{$\psi|_{L_A}$ is the pullback of $\varphi|_{L_B}$} via $\Theta|_{\widehat{L}_A}$.
\end{defn}

We say a pair $(G,\psi)$ is a $\calp$-by-$\Z$ group for some group property $\calp$ if $G$ admits an epimorphism $\psi\colon G\to \Z$ such that the kernel has property $\calp$.

\begin{prop}[$\widehat{\Z}$-regularity]\label{mu unit}
    Let $G_A$ and $G_B$ be $\{$type $\mathsf{FP_\infty}\}$-by-$\Z$ groups satisfying $b_1(G_A)=b_1(G_B)=1$.  If $\Theta\colon \widehat{G}_A\to\widehat{G}_B$ is an isomorphism, then there exists a unit $\mu\in\widehat\Z^\times$ such that $\MC(\Theta)=\mu\Z$.
\end{prop}
\begin{proof}
    By \cite[Proposition 3.2(1)]{Liu2023}, the $\Z$-module $\MC(\Theta)$ is a non-zero finitely generated free $\Z$-module spanned by the single entry of the $1\times1$ matrix $(\mu)$ over $\widehat{\Z}$. By \cite[Proposition 3.2(2)]{Liu2023} we obtain a homomorphism $\Xi \colon G_A^{\mathrm{fab}}\to G_B^{\mathrm{fab}}$ such that $\Psi_\ast=\mu\widehat{\Xi}$.  Moreover, $\mu$ is a unit because $\Theta$ is an isomorphism.  Hence, $\MC(\Theta_\ast)=\mu\Z$.
\end{proof}

\begin{prop}[Fibre closure isomorphisms]\label{fibre closure iso}
Let $(L_A,\psi)$ and $(L_B,\varphi)$ be $\{$type $\mathsf{FP}_\infty\}$-by-$\Z$ groups.   Suppose $\Theta\colon \widehat{L}_A\to\widehat{L}_B$ is an isomorphism and $\psi$ is the pullback of $\varphi$ via $\Theta$ with unit $\mu$.  If $F_A$ is the fibre subgroup of $L_A$, then $\overline{F}_A$ projects isomorphically onto $\overline{F}_B$, the closure of the fibre subgroup of $L_B$, under $\Theta$.
\end{prop}
\begin{proof}
By our definition of a pullback (\Cref{defn:Zhatreg}) there are two cases to consider: The first case is when $\Theta$ is a $\widehat{\Z}$-regular isomorphism; the second case is when we are given (by the pullback hypothesis) the following situation: $L_A$ and $L_B$ are finite index subgroups of groups $G_A$ and $G_B$ respectively such that there is $\widehat{\Z}$-regular isomorphism $\tilde\Theta\colon\widehat{G}_A\to\widehat{G}_B$ and $\psi$ is the pullback of $\varphi$ via $\tilde\Theta$.

We first prove the case where $\Theta$ is $\widehat{\Z}$-regular.  Our proof in this case essentially follows \cite[Corollary 6.2]{Liu2023}.  Write $L_A = F_A\rtimes_{\Psi} Z_A$ and $G_B = F_B\rtimes_{\Phi} Z_B$ with $Z_A\cong Z_B\cong \Z$.  Identify, $H_A$ with $G_A^{\mathrm fab}$ and $H_B$ with $G_B^{\mathrm fab}$.  By hypothesis the map $\Theta_\ast$ is the completion of an isomorphism $\Theta_\mu\colon H_A\to H_B$ followed by multiplication by $\mu$ in $\widehat{H}_B=H_B\otimes_Z \widehat{\Z}$.  Thus, $\psi$ is the composite
\[\begin{tikzcd}
H_A \arrow[r,"\Theta_\mu\otimes\mu"] & H_B\otimes_\Z\mu\Z \arrow[r, "1\otimes \mu^{-1}"] & H_B\otimes_\Z \Z \arrow[r, "="] & H_B \arrow[r, "\varphi|_{Z_B}"] & \Z.
\end{tikzcd}\]
We obtain that $\Theta_\ast(\ker\psi_\ast)=\mu F_\mu (\ker\varphi_\ast)=\mu\ker(\varphi_\ast)$ in $\widehat{H}_B$.  Since $\ker\varphi_\ast$ is a $\Z$-submodule of $H_B$, the closure of $\widehat{H}_B$ is invariant under multiplication by a unit.  Hence, $\Theta_\ast\overline{\ker\psi_\ast} = \overline{\mu \ker\varphi_\ast}= \mu \overline{\ker\varphi_\ast}=\overline{\ker\varphi_\ast}$.   This completes the proof of the first case.

We now prove the second case.  We may assume $G_A$ is a finite index overgroup of $H_A$ admitting a finite quotient $\alpha$ such that $\ker\alpha=H_A$.  Note that $\overline{F}_A$ is equal to the intersection of a finite index normal subgroup $\ker\widehat\alpha$ with $\ker\widehat{\tilde{\psi}}$ in $\widehat{G}_A$, where $\tilde\psi$ is the lift of $\psi$ to $G_A$.  Similarly, $\overline{F}_B=\ker\widehat\alpha\cap\ker\widehat{\tilde{\varphi}}$.  The result now follows from the $\widehat{\Z}$-regular case applied to $\tilde\Theta\colon G_A\to G_B$.
\end{proof}

Note that following proposition would be trivial if the unit $\mu$ equalled $1$.  However, the definition of pullback we are using (\Cref{defn:Zhatreg}) only assumes the existence of a unit.

\begin{prop}[Isomorphism of fibre subgroups]\label{fibre iso}
Let $(G_A,\psi)$ and $(G_B,\varphi)$ be free-by-cyclic groups.  Suppose $\Theta\colon \widehat{G}_A\to\widehat{G}_B$ is an isomorphism.  If $\psi$ is the pullback of $\varphi$ via $\Theta$,  then the fibre subgroup $F_A$ of $G_A$ and the fibre subgroup $F_B$ of $G_B$ are isomorphic.
\end{prop}
\begin{proof}
We will show that the degree of the first Alexander polynomials of $G_A$ and $G_B$ are equal.  By \Cref{ordAP is Bno} this computes the rank of the $\FF_p$-homology of $F_A$ and $F_B$ which determines their rank.  Since $F_A$ and $F_B$ are free groups this determines them up to isomorphism.  

Let $\psi_n$ and $\varphi_n$ denote the modulo $n$ reduction of $\psi\colon G_A\twoheadrightarrow\Z$ and $\varphi\colon G_B\twoheadrightarrow \Z$ respectively, namely the composites
\[\begin{tikzcd}
G_A\arrow[r, two heads, "\psi"] & \Z \arrow[r, two heads] & \Z/n \quad \quad \text{and} \quad \quad G_B\arrow[r, two heads, "\varphi"] & \Z \arrow[r, two heads] & \Z/n.
\end{tikzcd}\]
We endow $M_{A,n}\coloneqq\FF_p[\Z/n]$ with the $G_A$-module structure given by $\psi_n$ and $M_{B,n}\coloneqq \FF_p[\Z/n]$ with the $G_B$-module given by $\varphi_n$.  Since $G_A$ and $G_B$ are cohomologically good (\Cref{Serre good}), by \cite[Proposition~4.2]{BoileauFriedl2020} we have isomorphisms
$H_k(G_A;M_{A,n})\cong H_k(G_B;M_{B,n})$
for all $k,n\geq 0$.  In particular, $\dim_{\FF_p} H_k(G_A;M_{A,n}) = \dim_{\FF_p} H_k(G_B;M_{B,n})$.  Now, by applying \cite[Proposition~3.4]{BoileauFriedl2020} twice we get
\begin{align*}
    \deg\Delta_{G_A,1}^{\psi,\mathbf{1}}(t) &= \max_{n\in\NN}\left\{\dim_{\FF_p}H_1(G_A;M_{A,n}) - \dim_{\FF_p}H_0(G_A;M_{A,n}), \right\}\\
    &=\max_{n\in\NN}\left\{\dim_{\FF_p}H_1(G_B;M_{B,n}) - \dim_{\FF_p}H_0(G_B;M_{B,n}), \right\}\\
    &=\deg\Delta_{G_B,1}^{\varphi,\mathbf{1}}(t).\qedhere
\end{align*}
\end{proof}

\section{Profinite invariance of twisted Reidemeister torsion}\label{sec:Rtorsion}
The goal of this section is to establish profinite invariance of twisted Reidemeister torsion (\Cref{profinite R Torsion}) for free-by-cyclic groups with first Betti number equal to one.  We do this by first establishing invariance of the twisted Alexander polynomials in a more general setting.  Finally, in \Cref{sec.homostretch} we establish profinite invariance of homological stretch factors.

We record the following lemma to  show the reader that in the case of free-by-cyclic the assumption of \emph{good} is satisfied.  Note that it is a special case of \cite[Corollary~2.9]{Lorensen2008}.

\begin{lemma}\label{Serre good}
Let $G$ be a free-by-cyclic group.  Then $G$ is cohomologically good.
\end{lemma}

\subsection{Twisted Alexander polynomials}\label{sec:AP.profinite}

\begin{prop}[Profinite invariance of twisted Alexander polynomials]\label{profinite AP}
Let $(G_A,\psi_A)$ and $(G_B,\varphi_B)$ be residually finite $\{$good type $\mathsf{F}\}$-by-$\Z$ groups. Let $\Theta\colon\widehat{G}_A\to\widehat{G}_B$ be an isomorphism and suppose $\psi_A$ is the pullback of $\varphi_B$ via $\Theta$ with unit $\mu$. Let $\psi_B \in H^1(G_B, \Z)$ be a primitive fibred class. Let $\psi_A \in H^1(G_A, \Z)$ be the fibred class $\Theta^{\ast}_{\mu}(\psi_B)$. Fix a $\Theta$-corresponding pair of finite quotients  $\gamma_A \colon G_A \to Q$ and $\gamma_B \colon G_B \to Q$. Suppose $\rho \colon Q \to \mathrm{GL}(k, \mathbb{Q})$ is a representation and $\rho_A \colon G_A \to \mathrm{GL}(k, \mathbb{Q})$ and $\rho_B \colon G_B \to \mathrm{GL}(k, \mathbb{Q})$ the pullbacks. Then,
\[\Delta_{G_A,n}^{\psi_A,\rho_A}(t)\cdot\Delta_{G_A,n}^{\psi_A,\rho_A}(t^{-1})\doteq\Delta_{G_B,n}^{\varphi_B,\rho_B}(t)\cdot\Delta_{G_B,n}^{\varphi_B,\rho_B}(t^{-1}) \]
holds in $\QQ[t^{\pm1}]$ up to monomial factors with coefficients in $\QQ^\times$.
\end{prop}

Before proving \Cref{profinite AP} we will collect a number of facts. The following criterion is due to Ueki \cite[Lemma~3.6]{Ueki2018}.

\begin{thm}[Ueki]\label{Ueki criterion}
Let $a(t),b(t)\in \Z[t]$ be a pair of palindromic polynomials and $\mu\in\widehat{\Z}$ be a unit.  If the principal ideals $(a(t^\mu))$ and $(b(t))$ of the completed group algebra $\cgrZ$ are equal, then $a(t)\doteq b(t)$ holds in $\Z[t^{\pm1}]$.
\end{thm}

\begin{defn}[$\mu$-powers]
Let $G$ be a profinite group, let $g\in G$, and let $\mu\in\widehat{\Z}$.  We define the  \emph{$\mu$-power} of $g$ to be $g^\mu=\varprojlim_N g^n\mod N$ where $N$ ranges over the inverse system of open normal subgroups of $G$ and $n\in\Z$ is congruent to $\mu$ modulo $|G/N|$.  Note that $hg^\mu h^{-1}=(hgh^{-1})^\mu$ for all $h\in G$.
\end{defn}

The following fact is classical, for convenience we cite Liu.

\begin{lemma}\emph{\cite[Lemma~7.6]{Liu2023}}\label{integralising finite rep}
Let $L$ be a finite group.  If $\rho\colon L\to \GL_k(\QQ)$ is a representation, then $\rho$ is conjugate to the representation $\sigma_\QQ$ over $\QQ$ given by extension of scalars of some representation $\sigma\colon L\to \GL_k(\Z).$
\end{lemma}

\begin{remark}\label{reduction to integral reps}
Combining \Cref{integralising finite rep} and \Cref{AP conjugacy} we may assume without loss of generality that the representation $\rho$ is equal to the extension of scalars of some integral representation $\sigma\colon Q\to\GL_k(\Z)$.  We denote by $\sigma_A\colon G_A\to\GL_k(\Z)$ the pullback $\gamma_A^\ast(\sigma)$ and similarly write $\sigma_B$ for $\gamma_B^\ast(\sigma)$.
\end{remark}

By \Cref{fibre closure iso} and \Cref{fully separable fibres} we have a commutative diagram with exact rows
\begin{equation}\label{eqn BIG diagram}
\begin{tikzcd}
1 \arrow[r, no head, tail] & F_A \arrow[r, tail] \arrow[d, tail]               & G_A \arrow[r, "\psi_A", two heads] \arrow[d, tail]                       & \Z \arrow[d, tail] \arrow[r, two heads] & 1 \\
1 \arrow[r, tail]          & \widehat{F}_A \arrow[r, tail] \arrow[d, "\Theta_F"] & \widehat{G}_A \arrow[r, "\widehat{\psi}_A", two heads] \arrow[d, "\Theta"] & \widehat{\Z} \arrow[d, "\mu"] \arrow[r] & 1 \\
1 \arrow[r, tail]          & \widehat{F}_B \arrow[r, tail]                     & \widehat{G}_B \arrow[r, "\widehat{\varphi}_B", two heads]                & \widehat{\Z} \arrow[r]                  & 1 \\
1 \arrow[r, tail]          & F_B \arrow[u, tail] \arrow[r, tail]               & G_B \arrow[r, "\varphi_B", two heads] \arrow[u, tail]                    & \Z \arrow[u, tail] \arrow[r]            & 1,
\end{tikzcd}
\end{equation}
where $\Theta_F=\Theta|_{\overline{F}_A}$ and $\Theta_F$, $\Theta$, and $\mu$ are isomorphisms.

We now write $G_A=F_A\rtimes\langle t_A\rangle$ with $\psi_A(t_A)=1$ and $G_B=F_B\rtimes \langle t_B\rangle$ with $\varphi_B(t_B)=1$.  Now \eqref{eqn BIG diagram} implies that $\Theta(t_A)$ is conjugate to the $\mu$-power $t^\mu_B$ of $t_B$ in $\widehat{G}_B$, up to multiplication by an element of $\widehat F_B$.  That is, there exists $h\in \widehat G_A$ and $k\in\widehat{F}_B$ such that $\Theta(t_A)^h=k t_B^{\mu}$.  In particular, $\widehat \varphi_B(\Theta(t_A))=\widehat\varphi_B(t_B^\mu)$.

Let $M_A$ be $\Z^k$ equipped with the $F_A$-module structure given by $\sigma_A|_{F_A}$ and similarly for $M_B$.  Note that $\psi_A$ and $\varphi_B$ induce automorphisms $\Psi_A$ of $F_A$ and $\Phi_B$ of $F_B$ (up to choosing an inner automorphism).  Moreover, $\Psi_A$ induces a $\Z$-linear isomorphism $\psi_{A,n}\colon H_n(F_A;M_A)\to H_n(F_A;M_A)$.  We note that the choices made here for picking group automorphisms $\Psi_A$ and $\Phi_B$ only depend on the outer automorphism class.  This is sufficient for us since these induce the same action on $H_n(F_A;-)$ resp. $H_n(F_B;-)$  It follows that $\psi_{A,n}$ only depends on $\sigma$ and $\psi_A$.  We obtain a commutative diagram of $\Z$-modules with exact rows
\begin{equation}\label{eqn Z module splitting}
    \begin{tikzcd}
0 \arrow[r, tail] & H_n(F_A;M_A)_{\mathrm{tors}} \arrow[d, "{\psi_{A,n}^{\mathrm{tors}}}"] \arrow[r, tail] & H_n(F_A;M_A) \arrow[d, "{\psi_{A,n}}"] \arrow[r, two heads] & H_n(F_A;M_A)_{\mathrm{free}} \arrow[d, "{\psi_{A,n}^{\mathrm{free}}}"] \arrow[r, two heads] & 0 \\
0 \arrow[r, tail] & H_n(F_A;M_A)_{\mathrm{tors}} \arrow[r, tail]                                           & H_n(F_A;M_A) \arrow[r, two heads]                           & H_n(F_A;M_A)_{\mathrm{free}} \arrow[r, two heads]                                           & 0.
\end{tikzcd}
\end{equation}
Note that after fixing bases we may consider $\psi_{A,n}^{\mathrm{free}}$ as a matrix in $\GL(H_n(F_A;M_A)_{\mathrm{free}})$.  Define
\begin{equation}\label{eqn PAn}
    P_{A,n}(t)\coloneqq \det_{\Z[t^{\pm1}]}\left(\mathbf{1}-t\cdot \psi_{A,n}^{\mathrm{free}}\right)
\end{equation}
and
\begin{equation}\label{eqn PBn}
    P_{B,n}(t)\coloneqq \det_{\Z[t^{\pm1}]}\left(\mathbf{1}-t\cdot \varphi_{B,n}^{\mathrm{free}}\right).
\end{equation}

The following lemma is \cite[Lemma~7.7]{Liu2023}.  The proof goes through verbatim once one assumes the kernels of $\psi_A$ and $\varphi_B$ are type $\mathsf{F}$.

\begin{lemma}\emph{\cite[Lemma~7.7]{Liu2023}}\label{Liu7.7}
Adopt the notation from \Cref{profinite AP}, \Cref{reduction to integral reps}, \eqref{eqn PAn}, and \eqref{eqn PBn}. We have $\Delta^{\psi_A,\rho_B}_{G_A,n}(t)\doteq P_{A,n}(t)$ and $\Delta^{\varphi_B,\rho_B}_{G_B,n}(t)\doteq P_{B,n}(t)$ in $\QQ[t^{\pm1}]$ up to monomials with coefficients in $\QQ^\times$.
\end{lemma}

The following lemma is \cite[Lemma~7.8]{Liu2023}.  The proof goes through verbatim once one assumes that the kernels of $\psi_A$ and $\varphi_B$ are type $\mathsf{F}$, that $F_A$ and $F_B$ are fully separable in $G_A$ and $G_B$ respectively (this is given by \Cref{fully separable fibres}), and that $F_A$ and $F_B$ are good.

\begin{lemma}\emph{\cite[Lemma~7.8]{Liu2023}}\label{Liu7.8}
Adopt the notation from \Cref{profinite AP}, \Cref{reduction to integral reps}, \eqref{eqn PAn}, and \eqref{eqn PBn}. For all $n$ we have an equality of principal ideals $(P_{A,n}(t^\mu))=(P_{B,n}(t))$ in $\cgrZ$.
\end{lemma}

\begin{proof}[Proof of \Cref{profinite AP}]
This follows from \Cref{Liu7.7,Liu7.8} and \Cref{Ueki criterion} after observing that the polynomials $\Delta_{G_A,n}^{\psi_A,\rho_A}(t)\cdot\Delta_{G_A,n}^{\psi_A,\rho_A}(t^{-1})$ and $\Delta_{G_B,n}^{\varphi_A,\rho_B}(t)\cdot\Delta_{G_B,n}^{\varphi_B,\rho_B}(t^{-1})$ are palindromic by \Cref{AP -phi recip}.
\end{proof}

\subsection{Twisted Reidemeister torsion}\label{sec:profRtorsion}
We now prove profinite invariance of twisted Reidemeister torsion for free-by-cyclic groups with first Betti number equal to one.

\begin{corollary}[Profinite invariance of twisted Reidemeister torsion]\label{profinite R Torsion}
Let $(G_A,\psi_A)$ and $(G_B,\varphi_B)$ be free-by-cyclic groups. Let $\Theta\colon\widehat{G}_A\to\widehat{G}_B$ be an isomorphism. Let $\varphi_B \in H^1(G_B; \Z)$ be a primitive fibred class and suppose $\psi_A$ is the pullback of $\varphi_B$ via $\Theta$.   Fix a $\Theta$-corresponding pair of finite quotients  $\gamma_A \colon G_A \to Q$ and $\gamma_B \colon G_B \to Q$. Suppose $\rho \colon Q \to \GL(k, \QQ)$ is a representation and $\rho_A \colon G_A \to \GL(k, \QQ)$ and $\rho_B \colon G_B \to \GL(k, \QQ)$ the pullbacks.  Then,
\[\{\tau^{\psi_A,\rho_A}_{G_A}(t),\tau^{-\psi_A,\rho_A}_{G_B} \}=\{\tau^{\varphi_B,\rho_B}_{G_B}(t),\tau^{-\varphi_B,\rho_B}_{G_B} \}. \]
\end{corollary}
\begin{proof}
By \Cref{profinite AP}, unique factorisation in $\QQ[t^{\pm1}]$, and \Cref{AP -phi recip} we obtain
\[S_{A,n}=\{\Delta_{G_A,n}^{\psi_A,\rho_A}(t),\Delta_{G_A,n}^{-\psi_A,\rho_A}(t)\}=\{\Delta_{G_B,n}^{\varphi_A,\rho_B}(t), \Delta_{G_B,n}^{-\varphi_B,\rho_B}(t)\}=S_{B,n}. \]
By \Cref{defn RT FbyZ} the relevant Alexander polynomials are concentrated in degrees $0$ and $1$.  By \Cref{zeroth AP palindromic} the sets $S_{A,0}$ and $S_{B,0}$ contain exactly one element up to $\doteq$-equivalence.  Finally, the result follows from \Cref{defn RT FbyZ}.
\end{proof}

\subsection{Profinite invariance of homological stretch factors}\label{sec.homostretch}

\begin{thm}[Profinite invariance of homological stretch factors]\label{thm:homostretch}
Let $(G_A,\psi)$ and $(G_B,\varphi)$ be free-by-cyclic groups.  If $\Theta\colon\widehat{G}_A\to\widehat{G}_B$ is an isomorphism and $\psi$ is the pullback of $\varphi$ via $\Theta$, then $\{\nu_{\psi}^+,\nu_{\psi}^-\}=\{\nu_{\varphi}^+,\nu_{\varphi}^-\}$.
\end{thm}
\begin{proof}
Denote the non-trivial primitive characters of $G_A$ by $\psi_A^{\pm}$ and the non-trivial primitive characters of $G_B$ by $\varphi_B^{\pm}$.  By \Cref{profinite AP} we have \[\Delta_{G_A,1}^{\psi_A^+,\mathbf{1}}(t)\cdot\Delta_{G_A,1}^{\psi_A^-,\mathbf{1}}(t)\doteq\Delta_{G_B,1}^{\varphi_A^+,\mathbf{1}}(t)\cdot\Delta_{G_B,1}^{\varphi_B^-, \mathbf{1}}(t)\] over $\QQ[t^{\pm1}]$.  Normalise the polynomials so that every term is a non-negative power of $t$ and the lowest term is $1$, and note that each of the four terms has the same degree.  Now, by unique factorisation in $\QQ[t^{\pm1}]$ we obtain the equality of sets \[S_A=\{\Delta_{G_A,1}^{\psi_A^+,\mathbf{1}}(t),\Delta_{G_A,1}^{\psi_A^-,\mathbf{1}}(t)\}=\{\Delta_{G_B,1}^{\varphi_A^+,\mathbf{1}}(t),\Delta_{G_B,1}^{\varphi_B^-, \mathbf{1}}(t)\}=S_B.\]
Now, since we are working over $\QQ$ the set $S_A$ [resp. $S_B$] is the set of characteristic polynomials for $(\psi^{\pm}_{A})_1$ [resp. $(\varphi_B^{\pm})_1$], that is, the set of characteristic polynomials for the induced maps on degree $1$ homology of the respective fibres.  In particular, the sets \[\{\nu_{\psi}^+,\nu_{\psi}^-\}\text{ and }\{\nu_{\varphi}^+,\nu_{\varphi}^-\}\]
can be computed by taking the modulus of the largest root of the Alexander polynomials in $S_A$ and $S_B$.  The desired equality follows.
\end{proof}

\section{Profinite invariance of Nielsen numbers}\label{sec:profNielson}

Let $X$ be a connected, compact topological space that is homeomorphic to a finite-dimensional cellular complex, with a finite number of cells in each dimension, and let $f \colon X \to X$ be a self-map. Recall from \cref{sec:Nielson} the definitions of the fixed point index $\mathrm{ind}_m(f; \mathcal{O})$ of $f^m$ at any point $p \in \mathcal{O}$, and the $m$-th Nielsen number $N_m(f)$ of $f$.

We will write $M_f$ to denote the mapping torus \[M_f = \frac{X\times [0,1]}{(f(x), 0) \sim (x, 1)}.\] Let $x_0 \in X$. We fix a path $\alpha \colon I \to X$ such that $\alpha(0) = f(x_0)$ and $\alpha(1) = x_0$. We identify $X$ with the fibre $X \times \{0\}$ in $M_f$ and write $\bar{x}_0$ denote the image of $x_0$ in $M_f$. We define $t \in \pi_1(M_f, \bar{x}_0)$ to be the loop obtained by concatenation of paths $\eta \cdot \alpha$, where 
$\eta_s = (x_0, s)$ for $s \in [0,1]$. The \emph{induced character} $\varphi \colon \pi_1(M_f) \to \Z$ maps every loop in $X$ based at $x_0$ to zero, and $\varphi(t) = 1$. 

Let $\zeta \colon \pi_1(M_f) \to \mathbb{Q}$ be any map that is constant on conjugacy classes. Then the \emph{$m$-th twisted Lefschetz number} of $f$ with respect to $\zeta$ is
\begin{equation}\label{eq:1} L_m(f; \zeta) = \sum_{\mathcal{O} \in \mathrm{Orb}_m(f)} \zeta(\mathrm{cd}(\mathcal{O}))\cdot \mathrm{ind}_m(f; \mathcal{O}).\end{equation}

For a finite-dimensional representation $\rho \colon \pi_1(M_f) \to \mathrm{GL}(k, R)$ of $\pi_1(M_f)$, let $\chi_{\rho} \colon \pi_1(M_f) \to R$ denote the trace map. We write $\mathrm{exp}(\cdot)$ to denote the formal power series,
\[\mathrm{exp}(x) = \sum_{k=0}^{\infty}\frac{x^k}{k!}.\]

\begin{thm}[{\cite{Jiang1996}, \cite[Lemma~8.2]{Liu2023}}]\label{Jiang} Let $\varphi \colon \pi_1(M_f) \to \Z$ denote the induced character. Suppose that $\mathbb{F}$ is a commutative field of characteristic 0 and that $\rho \colon \pi_1(M_f) \to \mathrm{GL}(k, \mathbb{F})$ is a finite-dimensional linear representation of $\pi_1(M_f)$. Then
\[\tau^{\rho, \varphi}_{\pi_1(M_f), \mathbb{F}[t^{\pm 1}]^k} \doteq \mathrm{exp} \sum_{m\geq 1} L_m(f; \chi_{\rho}) \frac{t^m}{m},\]
where the equality holds as rational functions in $t$ over $\mathbb{F}$, up to multiplication by monomial factors with coefficients in $\mathbb{F}^{\times}$.
\end{thm}

Let $Q$ be a finite group. 
We say two elements $g_1$ and $g_2$ in $Q$ are \emph{$\widehat{\mathbb{Z}}$-conjugate} if the cyclic groups $\langle g_1 \rangle$ and $\langle g_2 \rangle$ are conjugate in $Q$ (note that this is equivalent to the notion of $\widehat{\Z}$-conjugacy defined in \cite{Liu2023}).
This gives rise to an equivalence relation on the set $\mathrm{Orb}(Q)$ of conjugacy classes of $Q$. We write $\Omega(Q)$ to denote the resulting set of equivalence classes. For $\omega \in \Omega(Q)$, we let $\chi_{\omega} \colon \mathrm{Orb}(Q) \to \mathbb{Q}$ denote the characteristic function of $\omega$.

\begin{lemma}[{\cite[Lemma~8.5]{Liu2023}}]\label{Nielsen_counting} Fix $m \in \mathbb{N}$. Let $\gamma \colon \pi_1(M_f) \to Q$ be a quotient of $\pi_1(M_f)$ onto a finite group $Q$. Then,
\[N_m(f) \geq \# \{\omega \in \Omega(Q) \mid L_m(f; \gamma^{*}\chi_{\omega}) \neq 0\}.\] 
\end{lemma}

Note that by \eqref{eq:1}, for every $\omega \in \Omega(Q)$ such that $L_m(f, \gamma^{*}\chi_{\omega}) \neq 0$, there exists some $\mathcal{O} \in \mathrm{Orb}_m(f)$ such that $\mathrm{ind}_m(f, \mathcal{O}) \neq 0$ and \[ \begin{split}\gamma^{*}\chi_{\omega}(\mathrm{cd}(\mathcal{O})) &= \chi_{\omega}  \circ \gamma(\mathrm{cd}(\mathcal{O}))  \\
&\neq 0,
\end{split}
\]
which holds if and only if $\gamma(\mathrm{cd}(\mathcal{O})) \in \omega$. Hence the number of such elements in $\Omega(Q)$ is bounded above by the number of essential $m$-periodic orbits of $f$, which is exactly $N_m(f)$.

The following lemma is a strengthening of Lemma~8.6 in \cite{Liu2023}, however the proof follows from Liu's proof with only a slight modification. We provide a sketch for the convenience of the reader.

\begin{lemma}\label{Nielsen_equiv} Suppose that $\pi_1(M_f)$ is conjugacy separable. Then, for any $m \in \mathbb{N}$ there exists a finite quotient $Q_m$ of $\pi_1(M_f)$ such that 
\begin{equation} \label{Nielsen_equalities}
\begin{split}  N_m(f) &= \{\omega \in \Omega(Q_m) \mid L_m(f; \gamma^{*}\chi_{\omega}) \neq 0\}, \text{ and }\\ N_m(f^{-1}) &= \{\omega \in \Omega(Q_m) \mid L_m(f^{-1}; \gamma^{*}\chi_{\omega}) \neq 0\}.\end{split}
\end{equation}
\end{lemma}

\begin{proof} 
Let $G = \pi_1(M_f)$ and write $\varphi \colon G \to \Z$ to denote the induced character, $t \in G$ the stable letter and $K = \mathrm{ker}\varphi$ the fibre subgroup as before. Since $G$ is conjugacy separable, for each $m \geq 1$ there exists a finite quotient $\tilde{\pi}_m \colon G \to \tilde{Q}_m$, such that for all $m$-periodic orbits of $f$ and $f^{-1}$, the corresponding distinct conjugacy classes in $G$ are mapped to distinct conjugacy classes in $\tilde{Q}_m$. 

By the discussion directly following the statement of \cref{Nielsen_counting}, the inequality provided by \cref{Nielsen_counting} is achieved when the conjugacy classes corresponding to the essential $m$-periodic orbits of $f$ are mapped to distinct $\widehat{\Z}$-conjugacy classes in the finite quotient. Hence, it suffices to find a finite quotient $\pi_m \colon G \to Q_m$ such that $\tilde{\pi}_m$ factors through $\pi_m$, and which satisfies the following property. If $x_1$ and $x_2$ are two elements of $G$ which correspond to $m$-periodic orbits of $f$, or of $f^{-1}$, and if $\langle \pi_m(x_1) \rangle$ and $\langle \pi_m(x_2) \rangle$ are conjugate in $Q_m$, then in fact the elements $\pi_m(x_1)$ and $ \pi_m(x_2)$ are  conjugate in $Q_m$. This will then imply that $\tilde{\pi}_m(x_1)$ and $ \tilde{\pi}_m(x_2)$ are conjugate in $\tilde{Q}_m$, since $\tilde{\pi}_m$ factors through $\pi_m$. Hence $x_1$ and $x_2$ are conjugate in $G$, showing that the required property holds for $\pi_m$.

To construct $Q_m$ note that the $m$-periodic orbits of $f$ correspond to elements in the coset $Kt^m$ of $G$, and the $m$-periodic orbits of $f^{-1}$ to the elements in the coset $Kt^{-m}$. If $\widebar{K}$ and $\bar{t}$ are the images of $K$ and $t$ in a finite quotient of $G$, then the coset $\widebar{K}\bar{t}^m$ is invariant under conjugation by elements in the quotient group. Hence, it suffices to find $Q_m$ such that that the cyclic subgroups generated by $\bar{x}_1$ and $\bar{x}_2$, for any $x_1, x_2 \in Kt^m$, intersect $\widebar{K}\bar{t}^m$ exactly at $\bar{x}_1$ and $\bar{x}_2$, respectively. It will then follow that if $\langle\bar{x}_1 \rangle$ and $\langle \bar{x}_2 \rangle$ are conjugate, then $\bar{x}_1$ and $\bar{x}_2$ are conjugate. The details of this construction are spelled out in the proof of Lemma~8.6 in \cite{Liu2023}. \end{proof}

We will also need the following proposition from representation theory of finite groups (see e.g. \cite[Section~12.4]{Serre1977}). We refer the reader to \cite[Lemma~8.4]{Liu2023} for the proof of this result rephrased in the language of $\widehat{\Z}$-conjugacy classes.

\begin{prop}\label{basis}
Let $K$ be a finite group. The set of irreducible finite-dimensional characters of $K$ over $\mathbb{Q}$ forms a basis for the space of maps $\mathrm{Orb}(K) \to \mathbb{Q}$ which are constant on $\widehat{\Z}$-conjugacy classes of $K$.
\end{prop}

Let $X_A$ and $X_B$ be topological spaces as before, with self-maps $f_A \colon X_A \to X_A$ and $f_B \colon X_B \to X_B$. We write $G_A = \pi_1(M_{f_A})$ and $G_B = \pi_1(M_{f_B})$, and let $\psi_A \colon G_A \to \Z$ and $\varphi_B \colon G_B \to \Z$ be the induced characters. 

\begin{lemma}\label{Nielsen number invariance}
    Suppose that $G_A$ and $G_B$ are conjugacy separable. Let $\Theta \colon \widehat{G}_A \to \widehat{G}_B$ be an isomorphism such that for every $\Theta$-corresponding pair of finite quotients $\gamma_B \colon G_B \twoheadrightarrow Q$ and $\gamma_A \colon G_A \twoheadrightarrow Q$ (see \Cref{corresponding quotients}), and all representations $\rho \colon Q \to \mathrm{GL}(k, \mathbb{Q})$, we have
    \[\{\tau^{\psi_A,\rho \gamma_A}_{G_A},\tau^{-\psi_A,\rho \gamma_A}_{G_B} \}=\{\tau^{\varphi_B,\rho\gamma_B}_{G_B},\tau^{-\varphi_B,\rho \gamma_B}_{G_B} \}. \] 
    Then, for every $m \in \mathbb{N}$,
    \[ \{N_m(f_A) , N_m(f_A^{-1}) \} = \{ N_m(f_B), N_m(f_B^{-1})\}.\]
\end{lemma}

\begin{proof}
Let $m \in \mathbb{N}$. Invoke \Cref{Nielsen_equiv} to obtain a finite quotient $\gamma_B \colon G_B \to Q_m$ such that 
\[N_m(f_B^{\pm}) = \#\{\omega \in \Omega(Q_m) \mid L_m(f_B^{\pm}; \gamma_B^{*}\chi_{\omega}) \neq 0\}.\]
By \Cref{basis}, for every $\omega \in \Omega(Q_m)$, $\chi_{\omega}$ can be expressed uniquely as a $\mathbb{Q}$-linear combination $\chi_{\omega} = \sum_i \lambda_i \chi_{\rho_i}$, where each $\rho_i \colon Q_m \to \mathrm{GL}(k_i, \mathbb{Q})$ is an irreducible representation, and $\lambda_i \in \mathbb{Q}$. Hence 
\[L_m (f_B; \gamma_B^{\ast}\chi_{\omega}) = \sum_i \lambda_i L_m(f_B; \gamma_B^{*}\chi_{\rho_i}).\]
Let $\gamma_A$ be the map obtained by composing 
\[ G_A \xrightarrow{\iota} \widehat{G}_A  \xrightarrow{\widehat{\gamma_B}} Q,\]
where $\iota \colon G_A \to \widehat{G}_A$ is the natural inclusion. In particular, $\gamma_A$ and $\gamma_B$ are $\Theta$-corresponding, and thus by our assumption, for every representation $\rho_i \colon Q_m \to \mathrm{GL}(k_i, \mathbb{Q})$ we have that
    \[\{\tau^{\psi_A,\rho_i \gamma_A}_{G_A},\tau^{-\psi_A,\rho_i \gamma_A}_{G_A} \}=\{\tau^{\varphi_B,\rho_i\gamma_B}_{G_B},\tau^{-\varphi_B,\rho_i \gamma_B}_{G_B} \}. \] 
\smallskip

By \cref{Jiang} it follows that, up to multiplication by monomials in $t$,
\[\tau^{\psi_A,\rho_i \gamma_A}_{G_A} (t) \doteq  1 + L_1(f_A; \gamma_A^{*}\chi_{\omega})t + \sum_{i=2}^{\infty} a_i t^i,\]
where for every $i \geq 2$, the coefficient $a_i$ is of the form \[a_i = \frac{1}{i} L_i(f_A; \gamma_A^{*}\chi_{\omega}) + C_i,\]
with $C_i$ a constant term obtained from the numbers $L_k(f_A; \gamma_A^{*}\chi_{\omega})$, $k < i$. Similarly, 
\[\begin{split} \tau^{-\psi_A,\rho_i \gamma_A}_{G_A} (t) &\doteq  1 + L_1(f_A^{-1}; \gamma_A^{*}\chi_{\omega})t + \sum_{i=2}^{\infty} b_i t^i, \\  
b_i &= \frac{1}{i} L_i(f^{-1}_A; \gamma_A^{*}\chi_{\omega}) + D_i,\end{split} \]
and each $D_i$ is a constant term which only depends on the numbers $L_k(f_A^{-1}; \gamma_A^{*}\chi_{\omega})$, $k < i$. Note that the coefficients $a_i$ and $b_j$ are non-zero for only finitely many values of $i$ and $j$. Furthermore, the analogous equalities hold true for $\tau^{\varphi_B,\rho_i\gamma_B}_{G_B}$ and $\tau^{-\varphi_B,\rho_i \gamma_B}_{G_B}$. 
\smallskip 

Hence, by comparing the coefficients of the powers of $t$ in the expansions of the Redemeister torsions, it follows that for each $\rho_i$, \[\{L_m(f_B ; \gamma_B^{*}\chi_{\rho_i}), L_m(f_B^{-1} ; \gamma_B^{*}\chi_{\rho_i})\} = \{L_m(f_A ; \gamma_A^{*}\chi_{\rho_i}), L_m(f_A^{-1} ; \gamma_A^{*}\chi_{\rho_i})\}. \] Thus, \begin{gather*}L_m(f_B ; \gamma_B^{*}\chi_{\omega}) +L_m(f_B^{-1} ; \gamma_B^{*}\chi_{\omega}) = L_m(f_A ; \gamma_A^{*}\chi_{\omega}) + L_m(f_A^{-1} ; \gamma_A^{*}\chi_{\omega}), \text{ and}\\ 
L_m(f_B ; \gamma_B^{*}\chi_{\omega})L_m(f_B^{-1} ; \gamma_B^{*}\chi_{\omega}) = L_m(f_A ; \gamma_A^{*}\chi_{\omega})L_m(f_A^{-1} ; \gamma_A^{*}\chi_{\omega}).
\end{gather*}
Solving the above equations, we obtain
\[\{L_m(f_B ; \gamma_B^{*}\chi_{\omega}), L_m(f_B ; \gamma_B^{*}\chi_{\omega})\} = \{L_m(f_A ; \gamma_A^{*}\chi_{\omega}), L_m(f_A^{-1} ; \gamma_A^{*}\chi_{\omega})\}. \]
Now, 
\[\begin{split}
       N_m(f_B) + N_m(f^{-1}_B) &= \#\{\omega \in \Omega(Q_m) : L_m(f_B, \gamma_B^{*}\chi_{\omega}) \neq 0\}\\
      &     +\#\{\omega \in \Omega(Q_m) : L_m(f_B^{-1}, \gamma_B^{*}\chi_{\omega}) \neq 0\} \\
     &= \#\{\omega \in \Omega(Q_m) : L_m(f_A, \gamma_A^{*}\chi_{\omega}) \neq 0\}\\ 
     &    + \#\{\omega \in \Omega(Q_m) : L_m(f_A^{-1}, \gamma_A^{*}\chi_{\omega}) \neq 0\} \\ 
     &\leq N_m(f_A) + N_m(f_A^{-1}),
\end{split} \]
where the last inequality follows from Lemma~\ref{Nielsen_counting}. The same argument shows that $N_m(f_A) + N_m(f_A^{-1}) \leq N_m(f_B) + N_m(f_B^{-1})$. Hence $N_m(f_A) + N_m(f_A^{-1}) = N_m(f_B) + N_m(f_B^{-1})$. Similarly, we get that $N_m(f_B)\cdot N_m(f_B^{-1}) = N_m(f_A)\cdot N_m(f_A^{-1}).$ It follows that \[\{N_m(f_A), N_m(f_A^{-1})\} = \{N_m(f_B), N_m(f_B^{-1})\}.\qedhere\]\end{proof}

Combining Corollary~\ref{profinite R Torsion} with Lemma~\ref{Nielsen number invariance} and Proposition~\ref{train_track}, we obtain the following theorem.

\begin{thm}[Profinite invariance of Nielsen numbers and stretch factors]\label{Nielsen numbers equality}
Let $G_A$ and $G_B$ be conjugacy separable free-by-cyclic groups with an isomorphism $\Theta \colon \widehat{G}_A \to \widehat{G}_B.$ Let $\varphi_B \in H^1(G_B, \mathbb{Z})$ be primitive and fibred, and let $\psi_A \in H^1(G_A, \mathbb{Z})$ be the primitive fibred class which is the pullback of $\varphi_B$ via $\Theta$. Let $(f^{\pm}_A, \Gamma_A)$ and $(f_B^{\pm}, \Gamma_B)$ be the corresponding relative train track representatives with stretch factors $\lambda_{f_A^{\pm}}$ and $\lambda_{f_B^{\pm}}$, respectively. Then, for all $m \in \mathbb{N}$, 
\begin{gather*}\{N_m(f_A) , N_m(f_A^{-1}) \} = \{ N_m(f_B), N_m(f_B^{-1})\}, \text{ and }\\ \{\lambda_{f_A}, \lambda_{f_A^{-1}} \}= \{\lambda_{f_B}, \lambda_{f_B^{-1}}\}.\end{gather*}
\end{thm}

We now have everything we need to prove \Cref{thmx:invariants}.  Note this is a slightly more general formulation than in the introduction and this introduction version follows from the below and \Cref{mu unit}.

\medskip
\setcounter{thmx}{1}
\begin{thmx} \label{thmx:invariants}
%
    Let $G_A$ and $G_B$ be free-by-cyclic groups with a $\widehat \Z$-regular isomorphism $\Theta \colon \widehat{G}_A \to \widehat{G}_B.$ Let $\varphi_B \in H^1(G_B, \mathbb{Z})$ be primitive and fibred, and let $\psi_A \in H^1(G_A, \mathbb{Z})$ be the primitive fibred class which is the pullback of $\varphi_B$ via $\Theta$.  Let $F_A$ be the fibre of $\psi_A$ in $G_A$ and let $F_B$ be the fibre of $\varphi_B$ in $G_B$.  Then,
    \begin{enumerate}
        \item $\rank F_A=\rank F_B$; \label{thmB1}
        \item the homological stretch factors are equal $\{\nu^+_{\psi_A},\nu^-_{\psi_A}\}=\{\nu^+_{\varphi_B},\nu^-_{\varphi_B}\}$; \label{thmB2}
        \item the characteristic polynomials of the actions on the fibres are equal, $\{\Char\Psi^+_A,\Char\Psi^-_A\}\doteq\{\Char{\Phi^+_B},\Char{\Phi^-_B}\}$; \label{thmB3}
        \item for each representation $\rho\colon G_A\to \GL(n, \QQ)$ factoring through a finite quotient, the twisted Alexander polynomials $\{\Delta^{\psi_A,\rho},\Delta^{-\psi_A,\rho}\}\doteq\{\Delta^{\varphi_B,\rho}_n,\Delta^{-\varphi_B,\rho}_n\}$ and the twisted Reidemeister torsions $\{\tau^{\psi_A,\rho},\tau^{-\psi_A,\rho}\}=\{\tau^{\varphi_B,\rho},\tau^{-\varphi_B,\rho}\}$ over $\QQ$ are equal. \label{thmB4}
    \end{enumerate}
    Moreover, if $G_A$ and $G_B$ are conjugacy separable, (e.g. if $G_A$ and $G_B$ are hyperbolic), then $\widehat G$ also determines the Nielsen numbers of $\psi_A$ and $\varphi_B$ and the homotopical stretch factors $\{\lambda^+_{\psi_A},\lambda^-_{\psi_A}\}=\{\lambda^+_{\varphi_B},\lambda^-_{\varphi_B}\}$.
\end{thmx}
\begin{proof}
%
With this setup we have that
\cref{thmB1} is given by \Cref{fibre iso}; \cref{thmB2} is given by \Cref{thm:homostretch}; \cref{thmB3} follows from (4) and the fact that we can identify $\Char{\Phi^\pm}$ with $\Delta^{\pm\varphi,\mathbf{1}}_1$; \cref{thmB4} is given by \Cref{profinite AP}. The final statement follows by \Cref{Nielsen numbers equality}.
\end{proof}

\section{Almost profinite rigidity for free-by-cyclic groups}\label{sec.proofmain}

The aim of this section is to prove \Cref{thmx:Irr}. We reproduce the statement below.  Before we prove the theorem we collect some facts.

\begin{lemma}\label{lem:finite.monodromy}
    Let $G_A$ and $G_B$ be free-by-cyclic groups with finite and infinite order monodromies respectively.  Then, $\widehat{G}_A$ is not isomorphic to $\widehat{G}_B$.
\end{lemma}
\begin{proof}
    Suppose for contradiction that such an isomorphism exists. Note that since the monodromy of $G_B$ has infinite order, the center $Z(G_B)$ of $G_B$ is trivial. Let $G_A = F_m \rtimes_{\phi} \Z$ where $\phi$ represents a finite order outer automorphism. Clearly $m \geq 2$, otherwise $G_A$ is virtually abelian and $G_B$ is a virtually abelian free-by-cyclic group, which contradicts the fact that $G_B$ has trivial center. 
    
    Let $G_A' \leq G_A$ be a finite-index subgroup of $G_A$ so that $G_A'\simeq F_m\times \Z$. Then $\widehat{G}_A'\simeq \widehat{F}_m\times\widehat{\Z}$. Let $H$ be the image of $\widehat{G}_A'$ under the isomorphism $\widehat{G}_A \simeq \widehat{G}_B$. Then, $H \simeq \widebar{G}_B' \simeq \widehat{G}_B'$, for some finite-index subgroup $G_B' \leq G_B$. Since $Z(G_B')=\{1\}$ we have $Z(\widehat{G}_B')/\overline{Z(G_B')}=Z(\widehat{G}_A')\simeq \widehat{\Z}$.  By \cite[Theorem~7.2]{Lueck1994} we have $b_1^{(2)}(G_B')=b_1^{(2)}(G_A')=b_1^{(2)}(F_m\times \Z)=0$, where $b_1^{(2)}$ denotes the first $\ell^2$-Betti number.  It follows that the dense projection $\pi$ of $G_B'$ to $\widehat{F}_m\leq \widehat{G}_A'$ is not injective.  Indeed, otherwise, by \cite[Corollary~3.3]{BridsonConderReid2016}, we have \[0 = b_1^{(2)}(G_B')\geq b_1^{(2)}(F_m)=m-1\geq 1,\] which is a contradiction.  It follows that $G_B'$ intersects $\ker\pi \leq Z(\widehat{G}'_B)$ non-trivially.  But then, $Z(G_B')\neq \{1\}$ contradicting our original hypothesis.
\end{proof}

\begin{prop}\label{prop:finite.monodromy}
    Let $G$ be a free-by-cyclic group with finite order monodromy and $b_1(G)=1$.  Then, $G$ is almost profinitely rigid amongst free-by-cyclic groups and every free-by-cyclic group in the profinite genus of $G$ has finite order monodromy.
\end{prop}
\begin{proof}
    Let $G_A$ be a free-by-cyclic group with finite order monodromy and first Betti number equal to one, and suppose $G_B$ is a free-by-cyclic group profinitely isomorphic to $G_A$.  By \Cref{lem:finite.monodromy} we may assume $G_B$ has finite order monodromy.  Note $b_1(G_B)=1$.  Now, \Cref{thmx:invariants}\eqref{thmB1} implies that the (uniquely defined) fibre subgroups of $G_A$ and $G_B$ have the same rank --- say $n$.  Since, by \cite{CullerVogtmann1986}, $\Out(F_n)$ has only finitely many conjugacy classes of torsion subgroups, there are only finitely many possibilities for the isomorphism type of $G_B$.
\end{proof}

Recall that an outer automorphism $\Phi \in \mathrm{Out}(F_n)$ is said to be \emph{atoroidal} if there does not exist a non-trivial element $x \in F_n$ and $n \geq 1$ such that $\Phi^n$ preserves the conjugacy class of $x$.

The following proposition is a folklore result which can be traced back to the work of Bestvina--Handel, who proved it for fully irreducible elements of  $\mathrm{Out}(F_n)$ \cite[Theorem~4.1]{BestvinaHandel1992}. A careful proof in the more general setting of expanding free group endomorphisms can be found in the paper of Mutanguha \cite[Theorem~A.4]{Mutanguha2021}.  

\begin{prop}\label{geometric}
    Let $\Phi \in \mathrm{Out}(F_n)$ be an outer automorphism of $F_n$. Suppose that $\Phi$ is infinite-order irreducible and not atoroidal. Then $\Phi$ is induced by a pseudo-Anosov homeomorphim of a once-punctured surface.
\end{prop}

\medskip
\setcounter{thmx}{0}
\begin{thmx} \label{thmx:Irr}
    Let $G$ be an irreducible free-by-cyclic group.  If $b_1(G)=1$, then $G$ is almost profinitely rigid amongst irreducible free-by-cyclic groups.
\end{thmx}

\begin{proof}
    Let $G_A$ be a free-by-cyclic group with $b_1(G_A) = 1$ and irreducible monodromy $\Phi$. Let $G_B$ be another free-by-cyclic group with irreducible monodromy $\Psi$ and suppose that $\widehat{G}_A \cong \widehat{G}_B$. If the monodromy $\Psi$ has finite order, then we are done by \Cref{prop:finite.monodromy}.  
    
    Assume $\Psi$ has infinite order. Note that by \Cref{thmx:hyperbolicity}, $\Phi$ is atoroidal if and only if $\Psi$ is atoroidal.
    
    If $\Psi$ is not atoroidal, then by \Cref{geometric}, both $\Phi$ and $\Psi$ are induced by pseudo-Anosov homeomorphisms of compact surfaces. Thus, $G_A$ and $G_B$ are fundamental groups of compact hyperbolic 3-manifolds and the result holds by \cite[Theorem~9.1]{Liu2023}.     

    Finally, suppose that $\Phi$ is atoroidal.  Hence $G_A$ and $G_B$ are Gromov hyperbolic free-by-cyclic groups. By \cite{HagenWise2015}, $G_A$ and $G_B$ are virtually compact special, and thus by \cite{Minasyan2006} they are conjugacy separable. Furthermore, $b_1(G_B) = 1$ since Betti numbers are invariants of profinite completions. Thus by \Cref{mu unit}, the isomorphism $\widehat
{G}_A \to \widehat{G}_B$ is $\widehat
\Z$-regular. Hence by \Cref{Nielsen numbers equality}, the sets of stretch factors $\{\lambda_{\Phi}, \lambda_{\Phi^{-1}}\}$ of $\Phi^{\pm 1}$ and $\{\lambda_{\Psi}, \lambda_{\Psi^{-1}}\}$ of $\Psi^{\pm 1}$ are equal.  Moreover, again by \Cref{Nielsen numbers equality}, the ranks of the corresponding fibres are equal. The result now follows from \Cref{min}.
\end{proof}

\subsection{Applications}

We conclude this section with the applications of \Cref{thmx:Irr}, \Cref{thmx:invariants} and \cref{thmx:hyperbolicity}.

\smallskip
\setcounter{thmx}{3}
\begin{corx}\label{corx:PV}
    Let $G$ be a super irreducible free-by-cyclic group.  Then, every free-by-cyclic group profinitely isomorphic to $G$ is super irreducible.  In particular, $G$ is almost profinitely rigid amongst free-by-cyclic groups.
\end{corx}

\begin{proof}
    Let $H$ be a free-by-cyclic group and suppose $\widehat{H}\cong\widehat{G}$. As explained in \cite[Section~2]{GerstenStallings1991} $G$ being super irreducible is a property of the characteristic polynomial of the matrix $M\colon H_1(F_n;\QQ)\to H_1(F_n;\QQ)$ representing the action of $\Phi$ on $H_1(F_n;\QQ)$.  Thus, by \Cref{thmx:invariants} we see $H$ is super irreducible.  The result follows from \Cref{thmx:Irr}.
\end{proof}

\begin{corx} \label{corx:generic}
      Let $G$ be a random free-by-cyclic group.  Then, asymptotically almost surely $G$ is almost profinitely rigid amongst free-by-cyclic groups.
\end{corx}

\begin{proof}
    By \Cref{generic_properties}, every generic free-by-cyclic group $G$ is super-irreducible and has $b_1(G)=1$.  The result follows from  \Cref{corx:PV}.
\end{proof}

\begin{corx}\label{corx:F3}
    Let $G=F_3\rtimes\Z$.  If $G$ is hyperbolic and $b_1(G)=1$, then $G$ is almost profinitely rigid amongst free-by-cyclic groups.
\end{corx}

\begin{proof}
   We first prove $G$ is irreducible.  Suppose that this is not the case. Then $G$ has a subgroup isomorphic to either $\Z\rtimes \Z$ or $F_2\rtimes \Z$. But both possibilities would imply that $G$ contains a $\Z^2$ subgroup contradicting hyperbolicity.  Now let $H$ be a free-by-cyclic group and suppose that $\widehat{H}\cong\widehat{G}$.  By \Cref{thmx:hyperbolicity} we see $H$ is hyperbolic and by \Cref{thmx:invariants} we see that $H$ splits as $F_3\rtimes\Z$.  Thus, the previous paragraph implies $H$ is irreducible.  The result follows from \Cref{thmx:Irr}.
\end{proof}

\begin{corx} \label{corx:F2}
      Let $G=F_2\rtimes\Z$.  If $b_1(G)=1$, then $G$ is profinitely rigid amongst free-by-cyclic groups.
\end{corx}

\begin{proof}
    Let $H$ be a free-by-cyclic group and suppose $\widehat{H}\cong\widehat{G}$.  By \Cref{thmx:invariants} we see that $H\cong F_2\rtimes \Z$.   But each $F_2\rtimes \Z$ is profinitely rigid amongst groups of the form $F_2\rtimes \Z$ by \cite{BridsonReidWilton2017}.
\end{proof}

\begin{remark}
In fact, Theorems A - C apply within a wider class of groups than stated in the hypothesis; namely, we can consider the class of mapping tori of (possibly infinite rank) free group automorphisms (imposing irreducibility if the fibre is finitely generated).  The key point is that by \cite{FeighnHandel1999} any finitely generated group $G$ in this class is finitely presented and has $\-\chi(G)\leq 0$ with equality if and only if the fibre subgroup is finitely generated.  Now, $\chi(G)<0$ if and only if $b_1^{(2)}(G)>0$ by \cite[Theorem~6.80]{Lueck2002}, but the first $\ell^2$-Betti number is a profinite invariant amongst finitely presented groups \cite[Corollary~3.3]{BridsonConderReid2016}.  It follows no \{infinitely generated free\}-by-cyclic group $G$ is profinitely isomorphic to a \{finitely generated free\}-by-cyclic group.
\end{remark}

\section{Profinite conjugacy in \texorpdfstring{$\Out(F_n)$}{Out(Fn)}}\label{sec.proconjugacy}

In this section we show that the stretch factors of atoroidal elements of $\Out(F_n)$ are profinite conjugacy invariants.

\begin{defn}[Profinitely conjugate]
    Let $\Psi,\Phi\in\Out(F_n)$.  We say $\Psi$ and $\Phi$ are \emph{profinitely conjugate} if they induce a pair of conjugate outer automorphisms in $\Out(\widehat{F}_n)$.
\end{defn}

\smallskip

\begin{thmx}\label{thmx:procongruence}
    Let $\Psi\in\Out(F_n)$ be atoroidal.  If $\Phi\in\Out(F_n)$ is profinitely conjugate to $\Psi$, then $\Phi$ is atoroidal and $\{\lambda_\Psi,\lambda_{\Psi^{-1}}\}=\{\lambda_\Phi,\lambda_{\Phi^{-1}}\}$.  In particular, if $\Psi$ is additionally irreducible, then there are only finitely many $\Out(F_n)$-conjugacy classes of irreducible automorphisms which are conjugate with $\Psi$ in $\Out(\widehat{F}_n)$
\end{thmx}

\smallskip 

\begin{proof}
    The first result follows from applying \Cref{thmx:hyperbolicity}, \Cref{thmx:invariants}, and \Cref{prop:procongruency}, the latter of which is proved below.  The ``in particular'' then follows from \Cref{min}.
\end{proof}

\begin{defn}[Aligned isomorphism]
    Let $\Psi,\Phi\in\Out(F_n)$. Write $G_A = F_n \rtimes_{\Psi} \Z$ and $G_B = F_n \rtimes_{\Phi} \Z$ and let $\psi \colon G_A \to \Z$ and $\psi \colon G_B \to \Z$ be the induced characters.  We say that an isomorphism $\Theta\colon\widehat{G}_A\to \widehat{G}_B$ is \emph{aligned} if the following diagram commutes
    \[\begin{tikzcd}
        \widehat{G}_A \arrow[r,"\widehat\psi"] \arrow[d,"\Theta"] & \widehat{\Z} \arrow[d,"\id"]\\
        \widehat{G}_B \arrow [r,"\widehat\varphi"] & \widehat{\Z}.
    \end{tikzcd}\]
    Note that an aligned isomorphism realises $\psi$ as the pullback of $\varphi$ with respect to $\Theta$ with unit $1$ in the sense that $\Theta_\ast(\varphi)=\psi$.
\end{defn}

The following proposition follows \cite[Proposition~3.7]{Liu2023b}.

\begin{prop}\label{prop:procongruency}
    Let $\Phi,\Psi\in\Out(F_n)$.  The following are equivalent:
    \begin{enumerate}
        \item the profinite completions of the free-by-cyclic groups $G_A=F_n\rtimes_\Psi\Z$ and $G_B=F_n\rtimes_\Phi\Z$ are aligned isomorphic; \label{prop:procongruency.1}
        \item the outer automorphisms $\Phi$ and $\Psi$ are profinitely conjugate. \label{prop:procongruency.2}
    \end{enumerate}
\end{prop}
\begin{proof}
    In constructing $G_A$ and $G_B$ we have implicitly picked lifts of $\Phi$ and $\Psi$ to $\Aut(F_n)$ which abusing notation we have also denoted by $\Phi$ and $\Psi$.  Write $G_A=F_n\rtimes_\Psi\langle t_A\rangle$ and $G_B=F_n\rtimes_\Phi\langle t_B\rangle$.  Denote the images of $t_A$ and $t_B$ is $\Out(F_n)$ by $\tau_A$ and $\tau_B$.  Note $\widehat{G}_A=\widehat{F}_n\rtimes\widehat{\langle t_A\rangle}$ and similarly for $G_B$.  Denote the images of $\tau_A$ and $\tau_B$ in $\Aut(\widehat{F}_n)$ by $\widehat\tau_A$ and $\widehat\tau_B$ respectively.
    
    We now prove that \eqref{prop:procongruency.1} implies \eqref{prop:procongruency.2}.  Suppose there is an aligned isomorphism $\Theta\colon \widehat{G}_A\to\widehat{G}_B$ and denote its restriction to $\widehat{F}_n$ by $\Theta_F$.  We have $\Theta(t_A)=t_Bh$ for some $h\in\widehat{F}_n$.  Since $gt_A=t_At_A^{-1}gt_A=t_A\widehat\tau_A(g)$ we have $\Theta_F(g)t_Bh=t_Bh\Theta_0(\widehat\tau_A(g))$.  Let $I_h$ denote the inner automorphism given by conjugation by $h$.  We have $\Theta_F(g)t_B=t_BI_h(\Theta_F(\widehat\tau_A(g))$, and hence, $t_B\widehat\tau_B(\Theta_F(g))=t_BI_h(\Theta_F(\widehat\tau(g)))$ for all $g\in\widehat{F}_n$.  Hence, $\widehat\tau_B=I_h\Theta_F\widehat\tau_A\Theta^{-1}$.  It follows that $\widehat\tau_A$ and $\widehat\tau_B$ are conjugate when projected to $\Out(\widehat{F}_n)$.  Hence, $\Phi$ and $\Psi$ are profinitely conjugate.

    To show \eqref{prop:procongruency.2} implies \eqref{prop:procongruency.1} we reverse the previous calculation to obtain a group isomorphism $\widehat{G}_A\to\widehat{G}_B$.
\end{proof}

\section{Automorphisms of universal Coxeter groups}\label{sec:Wn}

Let $n \geq 2$ be an integer. The \emph{universal Coxeter group of rank $n$} is the free product $W_n$ of $n$ copies of $\Z /2$, 
\[W_n = \bigast_{i=1}^n \Z /2.\] 
A \emph{free basis} of $W_n$ is a collection of $n$ elements $a_1 ,\ldots, a_n$ of $W_n$ of order 2, such that 
\[W_n \cong \langle a_1 \rangle \ast \ldots \ast \langle a_n \rangle.\]

\subsection{Graphs of groups}\label{sec:graphofgroups} 

For further detail and careful proofs of the claims made in this section, the interested reader is referred to \cite{Lyman2022a}. We closely follow the notation established there.

A \emph{graph of groups $(\Gamma, \mathcal{G})$ with trivial edge groups} consists of a connected graph $\Gamma$ and an assignment of a group $\mathcal{G}_v$ to every vertex $v$ of $\Gamma$. The vertex $v$ is said to be \emph{essential} if $\mathcal{G}_v$ is non-trivial. To every graph of groups with trivial edge groups $(\Gamma, \mathcal{G})$ we associate a graph of spaces $X_{\mathcal{G}}$ constructed by attaching a $K(\mathcal{G}_v, 1)$ with a unique vertex $v_0$ to the corresponding vertex $v$ of $\Gamma$. For the sake of brevity, we will sometimes write $\mathcal{G}$ to denote the graph of groups $(\Gamma, \mathcal{G})$. After fixing a basepoint and a spanning tree in $\mathcal{G}$, and immediately suppressing their notation, we write $\pi_1(\mathcal{G})$ to denote the fundamental group of the graph of groups $\mathcal{G}$.

A \emph{morphism} $F$ between graphs of groups $(\Gamma, \mathcal{G})$ and $(\Lambda, \mathcal{H})$ consists of a pair of maps $(f, f_X)$ with the following properties. The first map $f \colon \Gamma \to \Lambda$ sends vertices to vertices, and edges to edge paths. The second map $f_X \colon X_{\mathcal{G}} \to X_{\mathcal{H}}$ is a map of spaces such that the following diagram commutes,
\begin{equation*}\label{}\begin{tikzcd}
    X_{\mathcal{G}} \arrow[r, "f_X"] \arrow[d] & X_{\mathcal{H}} \arrow[d] \\
    \Gamma \arrow[r, "f"] & \Lambda
\end{tikzcd}
\end{equation*}
The vertical maps are the retractions obtained by collapsing the vertex spaces to their basepoints. 

A \emph{homotopy} from the morphism $(f, f_X) \colon (\Gamma, \mathcal{G} ) \to (\Lambda, \mathcal{H})$ to $(f', f'_X) \colon (\Gamma, \mathcal{G} ) \to (\Lambda, \mathcal{H})$ is a collection of morphisms \[\{(f_s, f_{X,s}) \colon \mathcal{G} \to \mathcal{H} : s\in [0,1]\},\] such that $\{f_s\}$ is a homotopy from $f$ to $f'$, and $\{f_{X,s}\}$ is a homotopy from $f_{X}$ to $f_{X}'$.

A morphism $F \colon \mathcal{G} \to \mathcal{H}$ is a \emph{homotopy equivalence,} if there exists a morphism $F' \colon \mathcal{H} \to \mathcal{G}$ 
such that $F \circ F'$ and $F' \circ F$ are homotopic to the identity morphisms. Any homotopy equivalence $H \colon \mathcal{G} \to \mathcal{H}$ induces an isomorphism $H_{*} \colon \pi_1(\mathcal{G}) \to \pi_1(\mathcal{H})$.

We will use the term \emph{combinatorial graph} when we want to emphasise that we are considering a graph with no extra structure.

\subsection{Topological representatives of \texorpdfstring{$\mathrm{Out}(W_n)$}{Out(Wn)} and Nielsen numbers}\label{sec:traintrackWn}

For each $n \geq 2$, define the \emph{thistle with $n$ prickles} to be the graph of groups $\mathcal{T}_n$, where the underlying graph is a tree with one vertex of degree $n$ and $n$ vertices of degree 1, and where each edge and the central vertex are labelled by the trivial group, and where the leaves are labelled by $\Z / 2$. Once and for all, fix the basepoint $\ast$ of $\mathcal{T}_n$ to be the central vertex. Then, there is a natural identification $\pi_1(\mathcal{T}_n, \ast) \simeq W_n$, so that each standard generator of $W_n$ is identified with the path in $\mathcal{T}_n$ given by the concatenation $e \cdot x \cdot \bar{e}$, where $e$ is an edge in $\mathcal{T}_n$ with $i(e) = \ast$ and $x$ is the generator of the group associated to the vertex $\tau(e)$.

Let $\Phi \in \mathrm{Out}(W_n)$. The \emph{standard topological representative} of $\Phi$ is the homotopy equivalence $\rho \colon (\mathcal{T}_n, \ast) \to (\mathcal{T}_n, \ast)$ determined by $\Phi$ and the identification $\pi_1(\mathcal{T}_n, \ast) \simeq W_n$ as above. A \emph{topological representative} of $\Phi$ is a pair $(F, \mathcal{G})$ where $\mathcal{G}$ is a graph of groups together with a homotopy equivalence $\alpha \colon \mathcal{T}_n \to \mathcal{G}$, and $F\colon \mathcal{G} \to \mathcal{G}$ is a homotopy equivalence, such that the following diagram commutes up to homotopy
\[
\begin{tikzcd}
    \mathcal{T}_n \arrow[d, "\alpha"] \arrow[r, "\rho"] & \mathcal{T}_n  \arrow[d, "\alpha"] \\
    \mathcal{G} \arrow[r, "F"]& \mathcal{G}
\end{tikzcd}
\]
where $\rho \colon \mathcal{T}_n \to \mathcal{T}_n$ is the standard representative of $\Phi$. We assume that $f$ is locally injective on the interiors of the edges of $\Gamma$. When we talk of the \emph{transition matrix, maximal filtration} and \emph{exponential strata} of $(F, \mathcal{G})$, we are referring to those objects associated to the underlying graph map $(f, \Gamma)$ (see Section~\ref{TopRep}). In particular, the topological representative $(F, \mathcal{G})$ is said to be \emph{irreducible} if the maximal filtration of the underlying graph map $(f, \Gamma)$ has length one. 

Let $(F, \mathcal{G})$ be a topological representative of $\Phi \in \mathrm{Out}(W_n)$. An \emph{invariant forest} for the representative $(F, \mathcal{G})$, where $F = (f,f_X)$, is an $f$-invariant subgraph $\Gamma_0$ of the underlying graph $\Gamma$, such that each component $C$ of $\Gamma_0$ is a tree and the fundamental group of the sub-graph of groups corresponding to $C$ acts with a global fixed point on its Bass--Serre tree. A forest is said to be \emph{non-trivial} if it contains at least one edge.

The outer automorphism $\Phi \in \mathrm{Out}(W_n)$ is said to be \emph{irreducible}, if every topological representative $(F, \mathcal{G})$ of $\Phi$, where the underlying graph $\Gamma$ has no inessential valence-one vertices and no invariant non-trivial forests, is irreducible. The \emph{stretch factor} of $\Phi$ is the infimum of the stretch factors of irreducible topological representatives of $\Phi$. The outer automorphism $\Phi$ is \emph{fully irreducible} if $\Phi^k$ is irreducible for every $k \geq 1$.

There exists a theory of (improved) relative train track representatives for elements of $\mathrm{Out}(W_n)$ \cite{Lyman2022b} (see also \cite{CollinsTurner1994}, \cite{FrancavigliaMartino2018} and  \cite{Lyman2022a} for earlier results on train tracks on graphs of groups), which is completely analogous to that for elements in $\mathrm{Out}(F_n)$. As in the case of $\mathrm{Out}(F_n)$, the stretch factor of an irreducible outer automorphism $\Phi \in \mathrm{Out}(W_n)$, as defined in the previous paragraph, coincides with the stretch factor of any train track representative. The stretch factor of a general element $\Phi \in \mathrm{Out}(W_n)$ is defined to be the stretch factor of any relative train track representative. 

The proof of the following lemma is completely analogous to the proof of Proposition~\ref{train_track}.

\begin{lemma} Let $\Phi \in \mathrm{Out}(W_n)$ be an outer automorphism of $W_n$ with stretch factor $\lambda$. Let $(F, \mathcal{G})$ be a topological representative of $\Phi$, with underlying graph map $f$. Then 
\[ \lambda = \mathrm{lim}\,\mathrm{sup}_{m \to \infty} N_m(f)^{1/m}.\]
\end{lemma}

Before proceeding further, we take a detour to discuss irreducibility of matrices and graphs. 

Let $A \in M_n(\mathbb{Z})$ be a matrix with non-negative integer entries $a_{ij}$. We construct a directed graph $\Gamma_A$ associated to $A$, so that $\Gamma_A$ has $n$ vertices $\{v_1, \ldots, v_n \}$ and there exist $a_{ij}$ directed edges from $v_i$ to $v_j$, for every $i,j \leq n$. The directed graph $\Gamma_A$ is said to be \emph{irreducible}, if for any two vertices $u$ and $v$ of $\Gamma_A$, there exists a directed path from $u$ to $v$. The following is an elementary exercise.

\begin{lemma} The non-negative integer matrix $A$ is irreducible if and only if the associated graph $\Gamma_A$ is irreducible.
\end{lemma}

We now prove a crucial lemma on the irreducibility of degree-two covers of directed graphs. In what follows, when we say \emph{path} from $u$ to $v$, we will always mean a directed path. Given an oriented edge $e$ in an oriented graph $\Gamma$, we write $i(e)$ to denote the initial vertex of $e$ in $\Gamma$ and $t(e)$ the terminal vertex.

\begin{lemma}\label{index_2_irred}
Let $\Gamma$ be a directed graph on $n$ vertices, and let $\Gamma'$ be a degree-two cover of $\Gamma$. If $\Gamma$ is irreducible then either $\Gamma'$ is irreducible, or it has two connected components and each is isomorphic to $\Gamma$.

Furthermore, if $\Gamma'$ is irreducible then the Perron--Frobenius eigenvalues of $A_{\Gamma'}$ and $A_\Gamma$ are equal.
\end{lemma}

\begin{proof} Let $\{v_1, \ldots, v_n\}$ be the vertex set of $\Gamma$. Let $v_i^1$ and $v_i^2$ be the two lifts of $v_i$ in $\Gamma'$, and write $V_1 = \{v_i^1 \mid 1 \leq i \leq n\}$ and $V_2 = \{v_i^2 \mid 1 \leq i \leq n\}$. Let $N$ be the number of edges $e$ in $\Gamma'$ such that $i(e) \in V_1$ and $t(e) \in V_2$. We call such edges \emph{special}. We prove our result by induction on $N$.

If $N=0$ then the lemma is clearly true, since $\Gamma'$ has two connected components and each is isomorphic to $\Gamma$. 

Let $N\geq 1$ and suppose the lemma is true whenever the number of special edges is at most $N-1$. Let $\Gamma' \to \Gamma$ be a degree-two cover with $N$ special edges. Note that since $\Gamma$ is irreducible, for any vertices $v_i$ and $v_j$ of $\Gamma$, there exists a path $\gamma$ from $v_i$ to $v_j$. This path has two lifts $\gamma_1$ and $\gamma_2$ in $\Gamma'$ such that either 
\begin{enumerate}[label=\roman*)]
    \item $\gamma_1$ joins $v_i^1$ to $v_j^1$ and $\gamma_2$ joins $v_i^2$ to $v_j^2$; or
    \item $\gamma_1$ joins $v_i^1$ to $v_j^2$ and $\gamma_2$ joins $v_i^2$ to $v_j^1$.
\end{enumerate}
Hence to prove the lemma it suffices to show that there exists a path in $\Gamma'$ from $v_k^1$ to $v_k^2$, and a path from $v_k^2$ to $v_k^1$, for all $k$. 

Let $e_1$ be a special edge and suppose that $i(e_1) = v_i^1$ and $t(e_1) = v_j^2$, for some $i$ and $j$. Then $\Gamma'$ contains an edge $e_2$ such that  $i(e_2) = v_i^2$ and $t(e_2) = v_j^1$. Construct a graph $\Gamma''$ from $\Gamma'$ by replacing $e_1$ with the edge $e_1'$ which joins $v_i^1$ to $v_j^1$, and replacing $e_2$ with the edge $e_2'$ which joins $v_i^2$ to $v_j^2$. Note that $\Gamma''$ is a degree-two cover of $\Gamma$ with $N-1$ special edges. 

Suppose first that $N = 1$ and fix index $k \leq n$. Since $\Gamma$ is irreducible, there exists a path in $\Gamma$ from $v_k$ to $v_i$. Let $\gamma$ be a shortest such path. Then $\gamma$ has two lifts $\gamma_1$ and $\gamma_2$ in $\Gamma''$. Since $\Gamma''$ has zero special edges, $\gamma_1$ only crosses edges with both endpoints in $V_1$ and $\gamma_2$ only crosses edges with both endpoints in $V_2$ (possibly after swapping $\gamma_1$ and $\gamma_2$). Also by minimality of the length of $\gamma$, the lifts of $\gamma$ do not cross the edges $e_1'$ and $e_2'$. Hence the path $\gamma_1$ descends to a path in $\Gamma'$ joining $v_k^1$ to $v_i^1$. Similarly one constructs a path from $v_j^2$ to $v_{k}^2$ in $\Gamma'$. The concatenation of these two paths and the edge $e_1$ gives a path from $v_k^1$ to $v_k^2$. 

Now assume $N\geq 2$. Then $\Gamma''$ is irreducible and thus there exists a shortest path $\eta_1$ in $\Gamma''$ from $v_k^1$ to $v_i^1$, and a shortest path $\eta_2$ from $v_j^2$ to $v_k^2$. Since $i(e_1') = v_i^1$, any shortest path from $v_k^1$ to $v_i^1$ does not contain $e_1'$. Similarly, any shortest path from $v_j^2$ to $v_k^2$ does not contain $e_2'$. Hence $\eta_1$ and $\eta_2$ descend to paths in $\Gamma'$. The concatenation of these paths, together with the edge $e_1$ give rise to a path from $v_k^1$ to $v_k^2$.  Similarly, one constructs a path from $v_k^2$ to $v_k^1$. Hence the statement holds for $\Gamma'$. This proves the first part of the lemma.

To prove the statement about equality of Perron--Frobenius eiganvalues, suppose that $\Gamma'$ is irreducible. Relabel the vertices of $\Gamma'$ so that for each $i \leq n$, the vertices labelled by $i$ and $i+n$ in $\Gamma'$ are the two lifts of the $i^{th}$ vertex of $\Gamma$. Let $a_{ij}$ and $a'_{ij}$ denote the $(i,j)^{th}$ elements of $A_{\Gamma}$ and $A_{\Gamma'}$, respectively. Since $\Gamma'$ is a degree-two cover of $\Gamma$, it follows that for every $i,j \leq n$,
\begin{equation} \label{entries} a_{ij} = a'_{ij} + a'_{i(j+n)} = a_{(i+n)j}' + a_{(i+n)(j+n)}'. \end{equation}
Let $v_{pf}$ denote the Perron--Frobenius eigenvector of $A_{\Gamma}$ and let $\lambda$ be the Perron--Frobenius eigenvalue. Let $v_{pf}'$ be the vector obtained by concatenating two copies of $v_{pf}$. Then by \eqref{entries},
\[A_{\Gamma'}v_{pf}' = \lambda \cdot v_{pf}'.\] 
Hence the Perron--Frobenius eigenvalue of $A_{\Gamma'}$ is $\lambda$.\end{proof}

Let $W_n$ be the universal Coxeter group with a free basis $\{a_1, \ldots, a_n\}.$ There exists a homomorphism $W_n \twoheadrightarrow \mathbb{Z}/2$ which maps each generator $a_i$ to the non-trivial element of $\mathbb{Z}/2$. The kernel $K \leq W_n$ is the unique torsion-free  index-two subgroup of $W_n$ and thus it is independent of the choice of the free basis. Moreover, $K$ is isomorphic to the free group of rank $n-1$. 

Fix a preferred free basis $X$ of the free group $F_{n-1}$. Let $\iota_X \in \mathrm{Aut}(F_{n-1})$ denote the automorphism which acts by inverting each element of $X$. We call $\iota_X$ the \emph{hyperelliptic involution} of $F_{n-1}$ with respect to $X$. We will write $\iota$ to denote $\iota_X$ when $X$ is clear from the context. Let $[\iota]$ be the image of $\iota$ in $\mathrm{Out}(F_{n-1})$.

\begin{remark}
    For any two choices of free generating sets $X$ and $Y$ of the free group $F$, the outer classes of the hyperelliptic involutions $[\iota_X]$ and $[\iota_Y]$ are conjugate in $\mathrm{Out}(F)$ \cite[Lemma~6.1]{BregmanFullarton2018}. 
\end{remark}

\begin{defn}[\cite{BregmanFullarton2018}]\label{defn:hyperelliptic}
    The \emph{hyperelliptic automorphism group} $\mathrm{HAut}(F_{n-1})$ is the centraliser of $\iota$ in $\Aut(F_{n-1})$. The \emph{hyperelliptic outer automorphism group} $\mathrm{HOut}(F_{n-1})$ of $F_{n-1}$ is the centraliser of $[\iota]$ in $\mathrm{Out}(F_{n-1})$.
\end{defn}

There is a homomorphism $\rho \colon \mathrm{Aut}(W_n) \to \mathrm{Aut}(F_{n-1})$ induced by restricting each automorphism of $W_n$ to the characteristic subgroup $K \leq W_{n-1}$. By \cite[Section~2 ]{Krstic1992}, the map $\rho$ restricts to an isomorphism 
\[\rho \colon \mathrm{Aut}(W_n) \to x^{-1}\,\mathrm{HAut}(F_{n-1}) \, x,\]
for some $x \in \mathrm{Aut}(F_{n-1})$.
Furthermore, the image of the subgroup $\mathrm{Inn}(W_n)$ of inner automorphisms of $W_n$ under $\rho$ is contained in the subgroup $\mathrm{Inn}(F_{n-1}) \cdot \langle \iota \rangle \cap \mathrm{HAut(F_{n-1})}$. Hence there is an isomorphism \[ \mathrm{Aut}(F_{n-1}) / \mathrm{Inn}(F_{n-1}) \to \left.\mathrm{HAut}(F_{n-1}) \right/ (\mathrm{Inn}(F_{n-1}) \cdot \langle \iota \rangle \cap \mathrm{HAut}(F_{n-1})) \]
Moreover, it is easy to see that $\mathrm{HAut}(F_{n-1}) \cap \mathrm{Inn}(F_{n-1}) = 1$, and hence there is an injective map 
\[\mathrm{Out}(W_n) \hookrightarrow \mathrm{HOut}(F_{n-1}) / \langle [\iota] \rangle.\]

It follows that each outer automorphism $\Phi$ in $\mathrm{Out}(W_n)$ defines a coset $\bar{\Phi}\cdot \langle [\iota] \rangle $ in the quotient $\mathrm{Out}(F_{n-1}) / \langle [\iota] \rangle$. Hence, there is a well-defined map $\mathrm{Out}(W_n) \to \mathrm{Out}(F_{n-1})$ which sends $\Phi$ to the outer automorphism $\bar{\Phi}^2$, which we label by $\Phi_K \in \mathrm{Out}(F_{n-1})$, and call the outer automorphism of $F_{n-1}$ \emph{induced} by $\Phi \in \mathrm{Out}(W_n)$.

\begin{prop}\label{stretch factors of W_n auts}

    Let $n \geq 3$ and $\Phi \in \mathrm{Out}(W_n)$ be an outer automorphism with stretch factor $\lambda(\Phi)$. Then, the stretch factor of the induced outer automorphism $\Phi_K \in \mathrm{Out}(F_{n-1})$ is equal to $\lambda(\Phi)^2$.

\end{prop}

\begin{proof}
    Let $(F , \mathcal{G})$ be a bounded relative train track representative of $\Phi^2 \in \mathrm{Out}(W_n)$, where $\mathcal{G} = (\Gamma, \mathcal{G})$ is a graph of groups as before, with the vertex $v_0$ in $\Gamma$ acting as a basepoint, and $F = (f, f_X)$. Let $\{a_1, \ldots, a_n\}$ be a free basis of $W_n$ so that each vertex of the underlying graph $\Gamma$ of $\mathcal{G}$ is labelled by some $\langle a_i \rangle \cong \Z / 2$ or the trivial group. Note that $\Gamma$ is simply connected. Let $K = \langle a_1a_2, a_1a_3, \ldots, a_1a_n \rangle$. 
    
    As before, let $X_{\mathcal{G}}$ denote the graph of spaces associated to $\mathcal{G}$. In particular, we identify $W_n$ with $\pi_1(X_{\mathcal{G}}, v_0, \Gamma)$. Let $\pi \colon Y \to X_{\mathcal{G}}$ be the cover of $X_{\mathcal{G}}$ corresponding to the subgroup $K$. Let $\widetilde{X}$ be a connected lift of $X_{\mathcal{G}}$ to $Y$ with $\tilde{v}_0 \in \widetilde{X}$ a lift of the basepoint $v_0$. 
    
    Since $K$ is a characteristic subgroup, there is a lift of the map $f_X$ to a map $f_Y \colon Y \to Y$ which represents the induced outer automorphism $\Phi_K$.

    Since each $a_i$ is not an element of $K$, the unique length-one loop in $X_{\mathcal{G}}$ contained in the free homotopy class of $a_i \in \pi_1(X_\mathcal{G})$ lifts to an edge with distinct endpoints. The endpoints are the two vertices of $Y$ which project down to the essential vertex labelled by $a_i$.
    
    Note that the morphism $f$ preserves the set of essential vertices. Let $Y'$ be the space obtained from $Y$ by collapsing the edges which join the two lifts of each essential vertex, and the lifts of the two-cells. Then $Y'$ is homotopy equivalent to $Y$, and there is a map $f_{Y'} \colon Y'\to Y'$ which is homotopic to $f_Y$. It follows that $(f_{Y'}, Y')$ is a topological representative of $\Phi_K \in \mathrm{Out}(F_{n-1})$. Then, $Y'$ is a (combinatorial) graph which is obtained by doubling the underlying graph $\Gamma$ of $\mathcal{G}$ along the essential vertices. In particular, the incidence matrix of $f_{Y'}$ gives rise to a directed graph which is an index-two cover of the directed graph associated to the incidence matrix of $f$.

    The relative train track structure of $f$ lifts to a relative train track structure of $f_{Y'}$. If $S$ is a non-zero stratum of $\mathcal{G}$ with stretch factor $\lambda$, then by Lemma~\ref{index_2_irred}, its lift to $Y'$ is either an irreducible stratum with stretch factor $\lambda$ or two irreducible strata, each with stretch factor $\lambda$. Then $\lambda(\Phi_K) = \lambda(\Phi^2) = \lambda(\Phi)^2$.
    
\end{proof}

 \subsection{Profinite invariants and almost rigidity of \{universal Coxeter\}-by-cyclic groups}\label{sec:ProfiniteCoxeter}

A group $G$ is said to be $\{$\emph{universal Coxeter$\}$-by-cyclic} if it fits into the short exact sequence
\[1 \to W_n \to G \to \Z \to 1.\]

For the remainder of this section, we let $(G_A, \varphi)$ and $(G_B, \psi)$ denote \{universal Coxeter\}-by-cyclic groups with fibred characters $\varphi \colon G_A \to \Z$ and $\psi \colon G_B \to \Z$.  We write $G_A = W_n \rtimes_{\Phi} \Z$ and $G_B = W_m \rtimes_{\Psi} \Z$ to denote the splittings of $G_A$ and $G_B$ induced by the characters, and let $K_A \leq G_A$ and $K_B \leq G_B$ be the unique torsion-free
index-two subgroups of the fibres. Recall that there is a well-defined map 
$\mathrm{Out}(W_n) \to \mathrm{Out}(F_{n-1})$ which sends an outer automorphism class $\Phi$ represented by $\phi \in \mathrm{Aut}(W_n)$, to the the outer automorphism class of $\phi^2|_K$, where $K \leq W_n$ is the unique torsion-free index-two subgroup. We write $\Phi_K$ to denote the image of $\Phi$ under this map, and call it the outer automorphism of $F_{n-1}$ \emph{induced by} $\Phi$.

Fix some $t \in \varphi^{-1}(1)$ and $s \in \psi^{-1}(1)$, and let 
\begin{equation}\label{index_2_free_by_z}\begin{split} H_A &= \langle K_A , t^2 \rangle_{G_A} \cong K_A \rtimes_{\Phi_{K_A}} \Z,  \\ H_B &= \langle K_B , s^2 \rangle_{G_B} \cong K_B \rtimes_{\Psi_{K_B}} \Z.\end{split}\end{equation}
We write $\bar{\varphi}$ to denote the character $\varphi \colon G_A \to \Z$ restricted to the subgroup $H_A$, and define $\bar{\psi}$ similarly. We note that the characters $\bar{\varphi}$ and $\bar{\psi}$ induce the splittings \eqref{index_2_free_by_z}.

For a group $G$ and prime $p$ we denote its \emph{pro-$p$ completion} by $\widehat{G}^\mathbf{p}$.  Note this is exactly the inverse limit of the system of finite quotients of order a power of $p$.

\begin{prop}\label{prop.profinite.Wn}
    Let $(G_A,\varphi)$ and $(G_B,\psi)$ be \{universal Coxeter\}-by-cyclic groups, and suppose $\Theta\colon \widehat{G}_A\to\widehat{G}_B$ is an isomorphism.  The following conclusions hold:
    \begin{enumerate}
        \item $\Theta$ is $\widehat{\Z}$-regular;\label{prop.propfinite.Wn.1}
        \item $G_A$ and $G_B$ have isomorphic fibres; \label{prop.propfinite.Wn.2}
        \item the free-by-cyclic groups $(H_A,\bar{\varphi})$ and $(H_B,\bar{\psi})$ satisfy that $\bar{\varphi}$ is the pullback of $\bar{\psi}$ via $\Theta|_{\widehat{H}_A}$; \label{prop.propfinite.Wn.3}
        \item $G_A$ and $G_B$ are good. \label{prop.propfinite.Wn.4}
    \end{enumerate}
\end{prop}
\begin{proof}
    It is easy to see that $G_A$ and $G_B$ satisfy $b_1(G_A)=b_1(G_B)=1$.  Thus, \eqref{prop.propfinite.Wn.1} follows from \Cref{mu unit}.  Note that $b_1(W_n;\FF_2)=n$.  We may prove \eqref{prop.propfinite.Wn.2} by an identical argument to \Cref{fibre iso} but taking the twisted Alexander polynomials over $\FF_2$ instead of an arbitrary prime.  
    
The subgroups $H_A \leq G_A$ and $H_B \leq G_B$ have finite index in their respective overgroups, and are free-by-cyclic. Since goodness passes to finite index overgroups this proves \eqref{prop.propfinite.Wn.4}.
    
    Now, the group $H_A$ is the kernel of a map $\alpha\colon G_A\onto\Z/2$.  We see that $H_A$ is torsion-free and so its pro-$2$ completion has finite cohomological dimension, whereas $G_A$ has $2$-torsion so $\cd_2(\widehat{G}_A^{\mathbf 2})=\infty$ (see \cite[Section 1.1. and Proposition~11.1.5]{Wilson1998} for the definition of $\cd_2$ and the relevant facts).  Completing the map $\alpha$ to $\widehat{G}_A$ we obtain an induced map $\widehat{G}_B\onto\Z/2$ and hence a map $\beta\colon G_B\onto \Z/2$.  Now $\ker\beta$ is torsion-free since $\ker \widehat\beta\cong\ker\widehat\alpha$ and $\cd_2(\ker\widehat\alpha^\mathbf{2})$ is finite.  We have shown that $H_A$ and $H_B$ are profinitely isomorphic free-by-cyclic groups with monodromies $\bar\varphi$ and $\bar\psi$ respectively.  Since $\Theta$ is $\widehat{\Z}$-regular by \eqref{prop.propfinite.Wn.1}, it follows that $\bar\varphi$ is the pullback of $\bar\psi$ via $\Theta|_{\widehat H_A}$.
\end{proof}

\begin{thmx}\label{thmx:CoxeterStretchInvariance}
    Suppose that all free-by-cyclic groups with monodromy contained in $\mathrm{HOut}(F_n)$ (see \cref{defn:hyperelliptic}) for some $n$, are conjugacy separable.  
    
    Let $(G_A,\varphi)$ and $(G_B,\psi)$ be profinitely isomorphic \{universal Coxeter\}-by-cyclic groups. Let $\{\lambda_{A}^+,\lambda_{A}^-\}$ and $\{\lambda_{B}^+,\lambda_{B}^-\}$ be the stretch factors of $(G_A, \varphi)$ and $(G_B, \psi)$, respectively. Then
    \[ \{\lambda_{A}^+,\lambda_{A}^-\}=\{\lambda_{B}^+,\lambda_{B}^-\}. \]
\end{thmx}

\begin{proof}
    The groups $(G_A, \varphi)$ and $(G_B, \psi)$ have isomorphic fibres by \cref{prop.profinite.Wn} \cref{prop.propfinite.Wn.1}, and by \cref{prop.profinite.Wn} \cref{prop.propfinite.Wn.3}, the character $\bar{\varphi} \colon H_A \to \Z$ is the pullback of $\bar{\psi} \colon H_B \to \Z$ under a profinite isomorphism $\widehat{H}_A \to \widehat{H}_B$. Also, by assumption, $(H_A, \bar{\varphi})$ and $(H_B, \bar{\psi})$ are conjugacy separable free-by-cyclic groups. Hence by \Cref{Nielsen numbers equality}, the stretch factors associated to $(H_A, \bar{\varphi})$ and $(H_B, \bar{\psi})$ are equal. Thus by \cref{stretch factors of W_n auts} the stretch factors of $(G_A, \varphi)$ and $(G_B, \psi)$ are equal.
\end{proof}

\bibliographystyle{halpha}
\bibliography{refs.bib}

\end{document}

%% file: top.tex
\newcounter{commentcounter}

\usepackage{amsmath} 
\usepackage{amssymb} 
\usepackage{amsthm} 
\usepackage{stmaryrd} 
\usepackage[english]{babel} 
\usepackage[font=small,justification=centering]{caption} 
\usepackage[nodayofweek]{datetime}
\usepackage[shortlabels]{enumitem} 
\usepackage[T1]{fontenc} 
\usepackage[utf8]{inputenc} 
\usepackage{ifthen} 
\usepackage{mathabx} 
\usepackage{mathtools} 
\usepackage[dvipsnames]{xcolor} 
\usepackage[pdftex,  colorlinks=true, pagebackref=true]{hyperref} 
    \hypersetup{urlcolor=RoyalBlue, linkcolor=RoyalBlue,  citecolor=black}
\usepackage{setspace} 
\usepackage{tikz-cd} 
\usepackage{xfrac} 
\usepackage[capitalize]{cleveref} 
\renewcommand*{\backref}[1]{}
\renewcommand*{\backrefalt}[4]
{
    \ifcase #1
        No citation in the text.
    \or
        Cited on page #2.
    \else
        Cited on pages #2.
    \fi
}

\makeatletter
\@namedef{subjclassname@1991}{Mathematical subject classification 1991}
\@namedef{subjclassname@2000}{Mathematical subject classification 2000}
\@namedef{subjclassname@2010}{Mathematical subject classification 2010}
\@namedef{subjclassname@2020}{Mathematical subject classification 2020}
\makeatother

\newtheorem{thm}{Theorem}[section]
\newtheorem{lemma}[thm]{Lemma}
\newtheorem{corollary}[thm]{Corollary}
\newtheorem{prop}[thm]{Proposition}

\newtheorem{question}[thm]{Question}

\newtheorem{thmx}{Theorem}
\newtheorem{corx}[thmx]{Corollary}

\theoremstyle{definition}
\newtheorem{defn}[thm]{Definition}
\newtheorem{remark}[thm]{Remark}

\theoremstyle{plain}

    \newtheoremstyle{TheoremNum}
        {\topsep}{\topsep} 
        {\itshape} 
        {-0.25cm} 
        {\bfseries} 
        {.} 
        { }  
        {\thmname{#1}\thmnote{ \bfseries #3}}
    \theoremstyle{TheoremNum}
    \newtheorem{duplicate}{}

\newcommand*{\claimproofname}{My proof}

\DeclareMathOperator{\aster}{\text{\LARGE{\textasteriskcentered}}}

\DeclareMathOperator{\Aut}{\mathrm{Aut}}
\DeclareMathOperator{\Out}{\mathrm{Out}}

\DeclareMathOperator{\Fix}{\mathrm{Fix}}

\DeclareMathOperator{\ord}{\mathrm{ord}}

\newcommand{\calc}{{\mathcal{C}}}

\newcommand{\calp}{{\mathcal{P}}}

\newcommand{\GL}{\mathrm{GL}}

\DeclareMathOperator{\id}{id}
\newcommand{\onto}{\twoheadrightarrow}

\newcommand{\TAP}{\mathsf{TAP}}

\newcommand{\rank}{\mathrm{rank}}

\newcommand{\cd}{\mathrm{cd}}

\newcommand{\MC}{\mathrm{MC}}

\DeclareMathOperator{\Char}{\mathrm{Char}}

\def\Z{\mathbb{Z}}
\newcommand{\cgrZ}{\widehat{\Z}\llbracket t^{\widehat{\Z}}\rrbracket}

\newcommand{\NN}{\mathbb{N}}
\newcommand{\ZZ}{\mathbb{Z}}

\newcommand{\RR}{\mathbb{R}}
\newcommand{\QQ}{\mathbb{Q}}
\newcommand{\FF}{\mathbb{F}}

\usepackage{tikz}
\usetikzlibrary{arrows,quotes}
\tikzstyle{blackNode}=[fill=black, draw=black, shape=circle]